\documentclass{amsart}[12pt]

\usepackage[hidelinks]{hyperref}
\usepackage{verbatim}
\usepackage{ctable}
\usepackage{yfonts} %
\usepackage{amssymb} %
\usepackage{amsthm}
\usepackage{array}
\usepackage{booktabs}%
\usepackage{hhline}%
\usepackage{xy} %
\usepackage{epsfig}%
\usepackage{color}%
\usepackage{upgreek}
\usepackage[english]{babel}
\usepackage{epigraph}%
\usepackage{fancybox}%
\usepackage{shadow}
\usepackage{afterpage}
\usepackage{mathrsfs}
\usepackage{enumitem}
\usepackage{tabularx}
\usepackage{subcaption}
\usepackage{graphicx}
\usepackage{type1cm}
\usepackage{eso-pic}
\usepackage{color}
\usepackage{upgreek}
\usepackage[foot]{amsaddr}

\newtheorem{theorem}{Theorem}
\newtheorem{lemma}{Lemma}

\theoremstyle{definition}
\newtheorem{definition}{Definition}
\newtheorem{remark}{Remark}
\newtheorem{example}{Example}

\theoremstyle{plain}
\newtheorem{corollary}{Corollary}

\newcommand{\ol}{\overline}
\newcommand{\otoprule}{\midrule[\heavyrulewidth]}
\newcommand{\vt}{\vspace{.1cm}}
\newcommand{\vtt}{\vspace{.2cm}}

\newcommand{\R}{\mathbb{R} }
\newcommand{\q}{\mathbb{Q}_{\epsilon}^n }
\newcommand{\Z}{\mathbb{Z} }
\newcommand{\N}{\mathbb{N} }
\newcommand{\h}{\mathbb{H} }

\newcommand{\hf}{\mathbb{H}_{\mathbb F}^m}

\newcommand{\s}{\mathbb{S}}

\renewcommand{\rho}{\varrho}
\renewcommand{\theta}{\varTheta}
\renewcommand{\Theta}{\varTheta}
\renewcommand{\Sigma}{\varSigma}
\renewcommand{\tau}{\uptau}
\usepackage{amsmath}

\newcommand{\overbar}[1]{\mkern 1.5mu\overline{\mkern-1.5mu#1\mkern-1.5mu}\mkern 1.5mu}

\newcommand{\transv}{\mathrel{\text{\tpitchfork}}}
\makeatletter
\newcommand{\tpitchfork}{%
 \vbox{
 \baselineskip\z@skip
 \lineskip-.52ex
 \lineskiplimit\maxdimen
 \m@th
 \ialign{##\crcr\hidewidth\smash{$-$}\hidewidth\crcr$\pitchfork$\crcr}
 }%
}
\makeatother

\begin{document}

\title[Weingarten Hypersurfaces]
{Elliptic Weingarten Hypersurfaces \\ of Riemannian Products}
\author{R. F. de Lima \and A. K. Ramos \and J. P. dos Santos}
\address[A1]{Departamento de Matem\'atica - UFRN}
\email{ronaldo.freire@ufrn.br}
\address[A2]{Departamento de Matemática Pura e Aplicada - UFRGS}
\email{alvaro.ramos@ufrgs.br}
\address[A3]{Departamento de Matem\'atica - UnB}
\email{joaopsantos@unb.br}

\maketitle

\begin{abstract}
Let $M^n$ be either a simply connected space form
or a rank-one symmetric space of noncompact type.
We consider  \emph{Weingarten hypersurfaces} of $M\times\mathbb R$, which are
those whose principal curvatures $k_1,\dots ,k_n$ and angle function $\theta$
satisfy a relation
$W(k_1,\dots,k_n,\theta^2)=0,$ being
$W$ a differentiable function which is  symmetric with respect to $k_1,\dots, k_n.$
When $\partial W/\partial k_i>0$ on the positive cone of $\R^n,$ a strictly convex Weingarten
hypersurface determined by $W$ is said to be \emph{elliptic}.
We show that, for a certain class of
Weingarten functions $W,$
there exist rotational strictly convex Weingarten
hypersurfaces of $M\times\mathbb R$ which are either topological spheres or entire graphs over $M.$
We establish a Jellett-Liebmann-type theorem by showing that
a compact, connected and
elliptic Weingarten hypersurface of either $\mathbb S^n\times\mathbb R$ or $\mathbb H^n\times\mathbb R$
is a rotational embedded sphere. Other uniqueness
results for complete elliptic Weingarten hypersurfaces
of these ambient spaces are obtained.
We also obtain existence results for constant scalar curvature hypersurfaces
of $\mathbb S^n\times\mathbb R$ and $\mathbb H^n\times\mathbb R$ which are either rotational or invariant
by translations (parabolic or hyperbolic). We apply our methods to give new proofs of the main
results by Manfio and Tojeiro on the classification of constant sectional curvature hypersurfaces
of \,$\mathbb S^n\times\mathbb R$ and $\mathbb H^n\times\mathbb R.$
\end{abstract}

\section{Introduction}

Amongst the compact surfaces of Euclidean space $\R^3,$ round spheres are known to be unique
with respect to several types of curvature constraints. For instance,
they are the only compact connected embedded
surfaces which have either constant mean curvature or constant Gaussian curvature, as attested
by the classical theorems of Alexandrov and Hilbert--Liebmann, respectively.
More generally, these theorems apply to
\emph{elliptic Weingarten surfaces}, which
are those whose principal curvature functions $k_1,k_2$ satisfy a relation
\[
W(k_1,k_2)= 0,
\]
where $W$ is a symmetric differentiable function on a domain $\Gamma\subset\R^2,$
which satisfies the ellipticity condition $\partial W/\partial{k_i}>0$ on $W^{-1}(0)\subset\Gamma.$
In \cite{galvez-mira(jdg)}, Gálvez and Mira improved these results by showing that  round
spheres are the only elliptic Weingarten spheres \emph{immersed} in $\R^3,$ which
proved affirmatively  a long standing conjecture by Alexandrov.

Elliptic Weingarten surfaces (sometimes called \emph{special Weingarten} surfaces)
of $\R^3$ and other three-spaces have been considered in many works
\cite{chern, folha-penafiel, hartman-wintner, rosenberg-saearp}.
More recently, Gálvez and Mira \cite{galvez-mira} conducted a thorough
investigation of rotationally invariant elliptic Weingarten surfaces
of homogeneous three-manifolds with isometry group of dimension $4$ (which include the Riemannian products
$\h^2\times\R$ and $\s^2\times\R$).
There, they
established many deep results regarding existence and uniqueness
of certain rotational spheres in the class of elliptic Weingarten surfaces
of these three-manifolds.

Inspired by the work of Gálvez and Mira, in the present paper
we consider \emph{Weingarten hypersurfaces} of Riemannian products $M^n\times\R.$
They are defined here as those whose  principal curvatures $k_1\,, \dots, k_n$
and angle function $\theta$ satisfy
\[W(k_1,\dots,k_n,\theta^2)=0,\]
where $W$ is a differentiable function which
is symmetric with respect to $k_1,\dots, k_n.$
Such a hypersurface $\Sigma$ is then called a $W$-hypersurface of $M\times\R.$
If, in addition, $\Sigma$ is strictly convex and
$W$ satisfies the ellipticity condition $\partial W/\partial k_i>0$
(for all $i=1,\dots ,n$) on the positive cone of $\R^n,$
we say that $\Sigma$ is an \emph{elliptic Weingarten hypersurface.}

Our work primarily concerns existence  and uniqueness (in the elliptic case) of
Weingarten hypersurfaces of $M\times\R$
when $M$ is either a simply connected space form or a rank-one symmetric space
of noncompact type (i.e., one of the hyperbolic spaces $\hf$).
The reason for considering these particular manifolds
relies on the fact that their geodesic spheres are isoparametric, i.e., have
constant principal curvatures. As we shall see,
this property allows us to construct
Weingarten vertical graphs in $M\times\R$ whose level hypersurfaces
are concentric geodesic spheres of $M$ (we call such graphs \emph{rotational}).

Given a general  Weingarten function $W,$ we associate
to it a first order ODE which
involves the principal curvature functions of geodesic spheres of $M.$ Then, we
call $W$  \emph{$M$-admissible}  if this equation admits a solution
$\rho:[0,\delta)\rightarrow[0,1)$ satisfying certain conditions (see Definition~\ref{def-weigartenpair}).
Then  we show that, for such an $M$-admissible $W,$  there exists a
rotational complete $W$-hypersurface $\Sigma$ of $M\times\R$  which is
homeomorphic to either the $n$-sphere $\s^n$ or Euclidean space $\R^n$.
In the latter case (which does not occur if $M=\s^n$),
$\Sigma$ is an entire graph over $M$, and in the former case, $\Sigma$
is obtained from the connected sum of
a graph over a closed ball of $M$ with its reflection over a horizontal hyperplane
of $M\times\R.$

Next, we study constant scalar curvature hypersurfaces of $\q\times\R$,
where $n\ge 3$ and $\q$ denotes the simply connected space form $\q$
of constant sectional curvature $\epsilon=\pm 1$ (i.e., $\s^n$ and $\h^n$).
Based on the fact that such hypersurfaces are Weingarten,
we establish that, for all $c>\epsilon n(n-1),$ there exists a
properly embedded strictly convex rotational
hypersurface $\Sigma$ in $\q\times\R$ with constant scalar curvature $c,$
which is, in fact, of constant sectional curvature.
Such a $\Sigma$ is either a sphere (if $c>0$) or an entire graph (if $c\le 0$).
For $c>\epsilon n(n-1),$ we also obtain a one-parameter family
of properly embedded Delaunay-type
rotational $n$-annuli in $\q\times\R$ with constant
scalar curvature $c,$ which are not of constant sectional curvature.
An analogous one-parameter family of
non periodic rotational $n$-annuli in $\h^n\times\R$ is obtained as well.

Similar results hold for hypersurfaces of $\h^n\times\R$ which
are invariant by either parabolic or hyperbolic translations.
Specifically, we show that for any
constant $c\in[-n(n-1),0),$
there exist in $\h^n\times\R$ an entire graph over $\h^n$ (of constant sectional curvature)
and a hypersurface which is symmetric
with respect to a horizontal hyperplane, both of constant scalar curvature $c$ and invariant
by parabolic translations. For such values of $c,$ we also show that there exists a one parameter
family of hypersurfaces of $\h^n\times\R$ with constant scalar curvature $c$ which are invariant by
hyperbolic translations. All these translational
hypersurfaces are properly embedded and homeomorphic to $\R^n.$

As these results indicate,
hypersurfaces of constant sectional curvature appear naturally when we are dealing with
constant scalar curvature hypersurfaces. Considering this fact,
we apply the methods developed here
to provide new proofs for the main theorems by Manfio and Tojeiro \cite{manfio-tojeiro}
regarding the classification of hypersurfaces of constant sectional curvature of \,$\q\times\R.$
(On this matter, see also \cite{aledo1, aledo2}, where the case $n=2$ was considered.)

Regarding uniqueness of elliptic  Weingarten hypersurfaces,
we establish  a Jellett-Liebmann type theorem.
Namely, by means of the Maximum Principle, the Alexandrov reflection technique, and
the methods and results in \cite{delima, esp-gal-rosen, esp-silva, esp-rosen, oliveira-schweitzer},
we show that, for $n\ge 3,$ a compact connected strictly convex
elliptic Weingarten hypersurface immersed in
$\q\times\R$  is necessarily an embedded rotational sphere. It is also shown that,
assuming such a  hypersurface to be complete, instead of compact, the same conclusion
holds under the additional assumption that
its height function has a critical point, and that
its least principal curvature is  bounded away from zero.

The paper is organized as follows. In Section \ref{sec-preliminaries}, we set some notation and formulae.
In Section \ref{sec-graphs}, we discuss graphs of Riemannian products  $M\times\R$ over parallel hypersurfaces of $M.$
In Section \ref{sec-ellipticWeingarten}, we introduce Weingarten hypersurfaces of
$M\times\R,$ establishing a key lemma.
In Section \ref{sec-W-rotational}, we consider rotational
Weingarten hypersurfaces of  $\hf\times\R$ and $\q\times\R.$ In Section \ref{sec-symmetricCSC},
we deal with constant scalar curvature hypersurfaces of $\q\times\R$  which are invariant by
either rotational or translational isometries. In Section \ref{sec-uniqueness}, we establish the Jellett--Liebmann type
theorem we mentioned, together with other rigidity results for Weingarten hypersurfaces of
$\q\times\R.$
In the concluding Section \ref{sec-sectionalcurvature},
we consider hypersurfaces of constant sectional curvature of $\q\times\R.$

\vtt

\noindent
{\bf Acknowledgements.} The second author was partially supported by CNPq/Brazil.
We are indebt to A. Martinez for letting us
know about Gálvez and Mira paper \cite{galvez-mira}, as well
as for suggesting to us bringing (a part of) it to the $n$-dimensional context.
We are also grateful to J. A. Gálvez, who read a preliminary  version of the paper and
made valuable suggestions. Finally, we thank the  referees for the careful reading
and insightful comments.

\section{Preliminaries} \label{sec-preliminaries}

Given an orientable Riemannian manifold $M^n,$ $n\ge 2,$ we shall consider the Riemannian product
$M\times\R$ endowed with its standard metric
\[
\langle\,,\,\rangle=\langle\,,\,\rangle_{M}+dt^2.
\]

Let $\Sigma$ be an oriented hypersurface of $M\times\R.$ Write
$N$ for its unit normal field and $A$ for its shape operator with respect to
$N,$ so that
\[
AX=-\overbar\nabla_XN, \,\, X\in T\Sigma,
\]
where $\overbar\nabla$ denotes the Levi-Civita connection of $M\times\R$ and
$T\Sigma$ stands for the tangent bundle of $\Sigma$.
The principal curvatures of $\Sigma,$ that is,
the eigenvalues of the shape operator $A,$ will be denoted by $k_1\,, \dots ,k_n$.
We shall say that $\Sigma$ is \emph{convex} (resp. \emph{strictly convex}) if, with respect to
a suitable orientation $N,$ $k_i\ge 0$ (resp. $k_i>0$) everywhere on $\Sigma,$ $i=1,\dots, n.$

The \emph{height function} $\xi$ and the \emph{angle function} $\Theta$ of $\Sigma$
are defined as
\[
\xi:=\pi_{\scriptscriptstyle\R}|_\Sigma \quad\text{and}\quad \Theta(x):=\langle N(x),\partial_t\rangle, \,\, x\in\Sigma,
\]
where $\partial _t$ denotes the gradient of the projection $\pi_{\scriptscriptstyle\R}$ of
$M\times\R$ on the second factor $\R.$
We denote the gradient of $\xi$ on $\Sigma$ by $T,$ which means that the equality
\begin{equation} \label{eq-gradxi}
T=\partial_t-\Theta N
\end{equation}
holds on $\Sigma.$ The trajectories of $T$ on $\Sigma$ will be called \emph{$T$-trajectories.}

Given $t\in\R,$ the set $P_t:=M\times\{t\}$ is called a \emph{horizontal hyperplane}
of $M\times\R.$ Horizontal hyperplanes are all isometric to $M$ and totally geodesic in
$M\times\R.$ In this context,
we call a transversal intersection $\Sigma_t:=\Sigma\transv P_t$ a \emph{horizontal section}
of $\Sigma.$ Any horizontal section $\Sigma_t$ is a hypersurface of $P_t$\,.
So, at any point $x\in\Sigma_t\subset\Sigma,$ the tangent space $T_x\Sigma$ of $\Sigma$ at $x$ splits
as the orthogonal sum
\begin{equation}\label{eq-sum}
T_x\Sigma=T_x\Sigma_t\oplus {\rm Span}\{T\}.
\end{equation}

The (first factor) manifolds $M$ we shall consider here are
the simply connected space forms
$\q$ of constant sectional curvature $\epsilon\in\{0,1,-1\},$ that is,
the Euclidean space $\R^n,$
the unit sphere $\s^n$ ($\epsilon=1)$ and the hyperbolic space $\h^n$ ($\epsilon=-1),$
as well as the rank-one symmetric spaces of noncompact
type, also known as the general hyperbolic spaces $\hf.$ Notice that the real hyperbolic
space $\h_\R^m$ is the standard hyperbolic space $\h^n$ of constant sectional curvature $-1$
(see, e.g., \cite{dominguez-vazquez}).

Let us recall that, denoting by $R$ and $\overbar R$ the curvature tensors of
$\Sigma$ and $\q\times\R,$ respectively,
for $X,Y, Z, W\in T\Sigma$, the Gauss equation reads as
\begin{equation}\label{eq-gauss001}
\langle R(X,Y)Z,W\rangle = \langle\overbar R(X,Y)Z,W\rangle+\langle AX,W\rangle\langle AY,Z\rangle-\langle AX,Z\rangle\langle AY,W\rangle,
\end{equation}
where $\overbar R$ vanishes identically for $\epsilon=0$ and, for $\epsilon=\pm 1,$  it  is given by (see \cite{daniel})
\begin{eqnarray} \label{eq-barcurvaturetensor}
 \epsilon\langle\overbar R(X,Y)Z,W\rangle &=& \langle X,W\rangle\langle Y,Z\rangle-\langle X,Z\rangle\langle Y,W\rangle \nonumber \\
 && +\langle X,Z\rangle\langle Y,\partial_t\rangle\langle W,\partial_t\rangle-\langle Y,Z\rangle\langle X,\partial_t\rangle\langle W,\partial_t\rangle \\
 && -\langle X,W\rangle\langle Y,\partial_t\rangle\langle Z,\partial_t\rangle+\langle Y,W\rangle\langle X,\partial_t\rangle\langle Z,\partial_t\rangle. \nonumber
\end{eqnarray}

\section{Graphs on Parallel Hypersurfaces} \label{sec-graphs}
Let $M_0^{n-1}$ and $M^n$ be two orientable Riemannian manifolds.
Assume that
\[f:M_0^{n-1}\rightarrow M^n\]
is an oriented embedding with unit normal field $\eta,$ and
suppose that there is an open interval $I\owns 0$
such that, for all $p\in M_0,$ the curve
\begin{equation}\label{eq-geodesic}
\gamma_p(s)=\exp_{\scriptscriptstyle M}(f(p),s\eta(p)), \, s\in I,
\end{equation}
is a well defined geodesic of $M$ without conjugate points. In this setting,
for any fixed $s\in I,$ the map
\[
\begin{array}{cccc}
f_s: & M_0 & \rightarrow & M\\
 & p & \mapsto & \gamma_p(s)
\end{array}
\]
is an embedding of $M_0$ into $M,$ which is said to be \emph{parallel} to $f.$
Observe that, given $p\in M_0$, the tangent space $f_{s_*}(T_p M_0)$ of $f_s$ at $p$ is the parallel transport of $f_{*}(T_p M_0)$ along
$\gamma_p$ from $0$ to $s.$ We also remark that, with the induced metric,
we will consider the unit normal $\eta_s$ of $f_s$ at $p$ given by
\[\eta_s(p)=\gamma_p'(s).\]

Having set the notation of the parallel immersions $f_s$, we introduce now the concept
of $(f_s,\phi)$-graph, which will play a fundamental role to prove some of the main results of the paper.

\begin{definition}
Let $\phi:I\rightarrow \phi(I)\subset\R$ be an increasing diffeomorphism, i.e., $\phi'>0.$
With the above notation, we call the set
\begin{equation}\label{eq-paralleldescription1}
\Sigma:=\{(f_s(p),\phi(s))\in M\times\R\,;\, p\in M_0, \, s\in I\},
\end{equation}
the \emph{graph} determined by $\{f_s\,;\, s\in I\}$ and $\phi,$ or $(f_s,\phi)$-\emph{graph}, for short
(Fig. \ref{fig-phigraph}).
\end{definition}

\begin{figure}[htbp]
\includegraphics{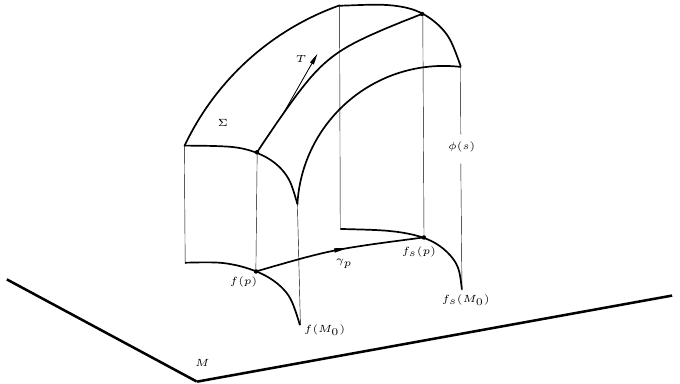}
\caption{An $(f_s,\phi)$-graph in $M\times\R.$}
\label{fig-phigraph}
\end{figure}

For an arbitrary point $x=(f_s(p),\phi(s))$ of an
$(f_s,\phi)$-{graph} $\Sigma,$ one has
\[T_x\Sigma=f_{s_*}(T_p M_0)\oplus {\rm Span}\,\{\partial_s\}, \,\,\, \partial_s=\eta_s+\phi'(s)\partial_t.\]
So, a unit normal to $\Sigma$ is
\begin{equation} \label{eq-normal}
N=\frac{-\phi'}{\sqrt{1+(\phi')^2}}\eta_s+\frac{1}{\sqrt{1+(\phi')^2}}\partial_t\,.
\end{equation}
In particular, its angle function is
\begin{equation} \label{eq-thetaparallel}
\Theta=\frac{1}{\sqrt{1+(\phi')^2}}\,\cdot
\end{equation}

As shown in \cite[Theorem 6]{delima-roitman}, any $(f_s,\phi)$-graph $\Sigma$
has the $T$-\emph{property}, meaning that
$T$ is a principal direction at any point of $\Sigma$.
More precisely, one has
\begin{equation}\label{eq-principaldirection}
AT=\frac{\phi''}{(\sqrt{1+(\phi')^2})^3}T.
\end{equation}

Given an
$(f_s,\phi)$-graph $\Sigma,$
let $\{X_1\,,\dots ,X_n\}$ be
an orthonormal frame of principal directions of $\Sigma$
in which $X_n=T/\|T\|.$ In this case,
for $1\le i\le n-1,$
the fields $X_i$ are all horizontal, that is,
tangent to $M,$ and constitute principal directions of the immersions $f_s$
at corresponding points (cf. \cite[Lemma 1]{delima-roitman})).
Therefore, setting
\begin{equation}\label{eq-rho}
\rho:=\frac{\phi'}{\sqrt{1+(\phi')^2}}
\end{equation}
and considering \eqref{eq-normal}, we have, for all $i=1,\dots ,n-1,$ that
\[
k_i=\langle AX_i,X_i\rangle=-\langle\overbar\nabla_{X_i}N,X_i\rangle=\rho\langle\overbar\nabla_{X_i}\eta_s,X_i\rangle=-\rho k_i^s,
\]
where $k_i^s$ is the $i$-th principal curvature of $f_s\,.$ Also,
it follows from \eqref{eq-principaldirection}
that $k_n=\rho'.$ Thus, the principal curvatures of the $(f_s,\phi)$-graph $\Sigma$
at $(f_s(p),\phi(s))\in\Sigma$ are
\begin{equation}\label{eq-principalcurvatures}
k_i=-\rho(s) k_i^s(p) \,\, (1\le i\le n-1) \quad\text{and}\quad k_n=\rho'(s).
\end{equation}

We remark that, up to a constant,
the function $\rho$ defined in \eqref{eq-rho} determines
the function $\phi.$
Indeed, it follows from equality \eqref{eq-rho} that
\begin{equation}\label{eq-phi10}
\phi(s)=\int_{s_0}^{s}\frac{\rho(u)}{\sqrt{1-\rho^2(u)}}du+\phi(s_0), \,\,\, s_0, s\in I.
\end{equation}

It should also be noticed that, from \eqref{eq-thetaparallel} and \eqref{eq-rho},
the unit normal $N$ defined in \eqref{eq-normal} can be
written as $N=-\rho\eta_s+\theta\partial_t$\,. Hence, the relation
\begin{equation} \label{eq-theta2}
\rho^2+\theta^2=1
\end{equation}
holds everywhere on any $(f_s,\phi)$-graph $\Sigma.$ In particular, $\rho=\|T\|$ on $\Sigma$.

\begin{definition}
A family $\mathscr F:=\{f_s:M_0\rightarrow M\,;\, s\in I\}$ of parallel
hypersurfaces is called \emph{isoparametric} if,
for each $s\in I,$ any principal curvature $k_i^s$ of $f_s\in\mathscr F$ is constant (possibly depending on
$i$ and $s$).
If so, each hypersurface $f_s$ is also called \emph{isoparametric.}
\end{definition}

It follows from \eqref{eq-principalcurvatures} that, if
$\mathscr F:=\{f_s:M_0\rightarrow M\,;\, s\in I\}$ is isoparametric and
$\Sigma$ is an $(f_s,\phi)$-graph in $M\times\R,$ then
all principal curvatures $k_i$ of $\Sigma$ at any point
$(f_s(p),\phi(s))$ are functions of $s$ alone.

\section{Elliptic Weingarten Hypersurfaces of $M\times\R$} \label{sec-ellipticWeingarten}

Let $\Gamma\subset\R^n$ be an open subset of $\R^n$ containing the \emph{positive cone}
\[\Gamma_+:=\{\mathbf{k}=(k_1\,,\dots ,k_n)\in\R^n\,;\, k_i > 0\}.\]
A symmetric function $W\in C^{\infty}(\Gamma)$
will be called a \emph{Weingarten function.}
If, in addition, $W$ satisfies the condition:
\begin{equation} \label{eq-condition}
\frac{\partial W}{\partial k_i}(\mathbf{k})>0 \,\,\,\, \forall\, \mathbf{k}\in\Gamma_+ \,\,\,\text{and}\,\,\, i=1,\dots ,n,
\end{equation}
then $W$ will be called an \emph{elliptic Weingarten function.}

\begin{example} \label{exam-basic}
Two distinguished elliptic Weingarten functions are the
following:
\begin{itemize}[parsep=1.5ex]
\item[i)] $W(k_1,\dots ,k_n)=\sum_{i_1<\cdots<i_r}k_{i_1}\dots k_{i_r}\,, \quad r\in\{1,\dots, n\}. $
\item[ii)] $W(k_1,\dots, k_n)=\sqrt{k_1^2+\cdots +k_n^2}.$
\end{itemize}
The function in (i) is the non normalized $r$-th mean curvature $H_r$ (notice that
$H_1=k_1+\cdots +k_n$ is the \emph{mean curvature} and $H_n=k_1\dots k_n$ is the \emph{Gauss-Kronecker curvature}), whereas
the function in (ii) is the \emph{norm of the second fundamental form} $\|A\|.$
We add that these functions are both homogeneous of degree one.
(Recall that, given $d\in\R,$ a function $f=f(\mathbf k)$ defined in a cone $\Gamma\subset\R^n$ is said to be \emph{homogeneous of degree $d$}
if $f(t\mathbf k)=t^df(\mathbf k)$ whenever $\mathbf k\in\Gamma$ and $t>0.$)
\end{example}

\begin{example} \label{exam-inverses}
Given  $W=W(k_1,\dots ,k_n)\in C^{\infty}(\Gamma_+),$ define
its \emph{inverse} $W^*$  as
\[
W^*(k_1,\dots ,k_n):={1}/{W(1/k_1\,, \dots ,1/k_n)}.
\]
It is easily checked that $W$ is homogeneous elliptic Weingarten if and only if
$W^*$ is homogeneous elliptic Weingarten.
\end{example}

Following \cite{galvez-mira}, given an open set $\Gamma\subset\R^n$ with
$\Gamma_+\subset\Gamma,$ we say that
\[
W=W(k_1,\dots ,k_n,\theta^2), \,\,\, (k_1\,, \dots, k_n,\theta^2)\in \Gamma\times[0,1],
\]
is a \emph{general Weingarten function} (resp. a \emph{general elliptic Weingarten function})
if, for any fixed $\theta\in [0,1],$ the map
\[
(k_1,\dots ,k_n)\in\Gamma \,\,\, \mapsto \,\,\, W(k_1,\dots,k_n,\theta^2)\in\R
\]
is a Weingarten function (resp. an elliptic Weingarten function).

\begin{definition}
We say that a hypersurface $\Sigma$ of a Riemannian product $M\times\R$  is a \emph{Weingarten hypersurface} if its
principal curvatures $k_1\,, \dots ,k_n$\,, together with its angle function $\theta,$  satisfy a relation
of the type
\begin{equation} \label{eq-generalweigarten}
W(k_1\,, \dots ,k_n,\theta^2)=0,
\end{equation}
where $W$ is a general  Weingarten function.
More specifically, we shall say that such a $\Sigma$ is a $W$-hypersurface.
If, in addition, $W$ is elliptic and
$\Sigma$ is strictly convex, then $\Sigma$ will be called an
\emph{elliptic Weingarten hypersurface}.
\end{definition}

Hypersurfaces of $M\times\R$ with  constant  mean curvature $H_r$
are canonical examples of  Weingarten hypersurfaces.
In the next section, we shall see that hypersurfaces of constant scalar curvature
in $\q\times\R$ are also  Weingarten hypersurfaces.

\begin{remark} \label{rem-MP}
We point out that the Hopf Maximum Principle applies to elliptic Weingarten hypersurfaces
(see \cite{folha-penafiel} for a
detailed discussion in the case $n=2$).
The same is true for the Continuation Principle, by the  results
in \cite{kazdan, protter}.
\end{remark}

The following lemma, which plays a fundamental role here,
characterizes  Weingarten $(f_s,\phi)$-graphs (with $\{f_s\}$ isoparametric)
as those whose associated $\rho$-functions are solutions of a certain first order
ordinary differential equation. In fact, it follows directly from the definition of Weingarten
hypersurface, the relations \eqref{eq-principalcurvatures}, and
equality \eqref{eq-theta2}.

\begin{lemma} \label{lem-graph}
Let $\Sigma$ be an $(f_s,\phi)$-graph in $M\times\R$ whose associated
family
\[\mathscr F:=\{f_s:M_0\rightarrow M\,;\, s\in I \}\]
of parallel hypersurfaces is isoparametric.
Then, given a Weingarten function $W\in C^\infty(\Gamma),$
we have that $\Sigma$ is a $W$-hypersurface of $M\times\R$ if and only if
its $\rho$-function satisfies the equality
\begin{equation}\label{eq-edo}
 W(-k_1^s\rho(s), \dots ,-k_{n-1}^s\rho(s),\rho'(s), 1-\rho^2)=0,
\end{equation}
where $k_1^s,\dots ,k_{n-1}^s$ are the principal curvatures of
$f_s\in\mathscr F.$
\end{lemma}

It follows from  Lemma \ref{lem-graph} and Picard's
theorem for local existence of solutions of first order ODE's that,
for any general Weingarten function $W,$
there exist  local $W$-hypersurfaces in
$M\times\R,$ provided that $M$ admits families of  isoparametric hypersurfaces.
In this context, it is natural to ask under which conditions we can construct complete
$W$-hypersurfaces in $M\times\R.$ We shall pursue this question in the next sections.

\section{Rotational  Weingarten Hypersurfaces of $\q\times\R$ and $\hf\times\R$} \label{sec-W-rotational}

Let $M^n$ be either a simply connected space form $\q$ or a rank-one symmetric space of
noncompact type $\hf.$ Take a point $o\in M$ and let $f_s:\s^{n-1}\rightarrow M$ be
the isometric immersion such that $f_s(\s^{n-1})$ is the geodesic sphere of
$M$ with center at $o$ and radius $s>0.$ In this setting, define
\begin{equation} \label{eq-isoparametricspheres}
\mathscr F:=\{f_s:\s^{n-1}\rightarrow M\,;\, s\in (0,\mathcal R_M)\},
\end{equation}
where $\mathcal R_M$ is given by
\begin{equation} \label{eq-radius}
\mathcal R_M:=\left\{
\begin{array}{ccl}
+\infty & \text{if} & M=\R^n\, \text{or} \,\,\hf,\\[1ex]
\pi/2   & \text{if} & M=\s^n.
\end{array}
\right.
\end{equation}

As is well known, $\mathscr F$ is isoparametric and each sphere $f_s(\s^{n-1})$ is
strictly convex (see \cite{dominguez-vazquez} and the references therein). In accordance
to the notation of Section \ref{sec-graphs}, for each $s\in (0,+\infty),$
we choose the outward orientation of $f_s$\,,
so that \emph{any principal curvature $k_i^s$ of \,$f_s$ is negative}.

In what follows, for $M$ and $\mathscr F$ as above,
we construct complete strictly convex  Weingarten hypersurfaces
in $M\times\R$ from $(f_s,\phi)$-graphs, $f_s\in\mathscr F.$ Since the elements of
$\mathscr F$ are concentric geodesic spheres, we shall call such a hypersurface
\emph{rotational}.

The general idea for this construction is to consider \eqref{eq-edo} as an ODE with variable $\rho.$
From a suitable solution to this equation, we obtain a function $\phi$ (using \eqref{eq-phi10})
which, by Lemma \ref{lem-graph},
defines a Weingarten graph $\Sigma'$ in $M\times\R$ over an open ball
$B_\delta(o)\subset M$ with $\delta\le +\infty.$ If $\delta=+\infty$, $\Sigma'$
is complete and we are done. Otherwise,
$\partial\Sigma'$ is an $(n-1)$-sphere in a horizontal hyperplane $P_t:=M\times\{t\},$ and
the tangent spaces of $\Sigma'$ along its boundary are all vertical (i.e., parallel to $\partial_t$).
Hence, a complete Weingarten $n$-sphere is obtained by ``gluing'' $\Sigma'$
with its reflection over $P_t$ along their common boundary.

As we shall see, the effectiveness of this procedure depends on the existence
of a solution $\rho$ to \eqref{eq-edo} which can be defined at the singular point $s=0.$
This fact leads us to introduce the concept of \emph{admissible} Weingarten function, as given below.
(Notice that the principal curvatures $k_i^s$ of the spheres $f_s$ are not defined at $s=0.$)

\begin{definition} \label{def-weigartenpair}
Let $M$ and $\mathscr F$ be as above.
We say that  a general  Weingarten function $W$
is $M$-\emph{admissible} (or simply \emph{admissible}) if equation \eqref{eq-edo}
has a solution $\rho$ defined in $[0,\delta),$ $0<\delta\le\mathcal R_M$, which satisfies the conditions:
\begin{itemize}[parsep=1ex]
\item[(C1)] $\rho(0)=0.$
\item[(C2)] $0<\rho<1$ on $(0,\delta).$
\item[(C3)] $\rho'>0$ on $(0,\delta).$
\item[(C4)] The limits
$\displaystyle\lim_{s\rightarrow 0}\rho(s)k_i^s$ and $\displaystyle\lim_{s\rightarrow 0}\rho'(s)$
exist and are finite (recall that
$k_i^s$ is the $i$-th principal curvature of $f_s$) and,
if $\delta<\mathcal R_M$, the limit
$\displaystyle\lim_{s\rightarrow\delta}\rho'(s)$
also exist and is finite.
\end{itemize}
We assume that $\delta$ is maximal with respect to (C2)--(C3) and call
$\rho:[0,\delta)\rightarrow[0,1)$ an \emph{associated function} to  $W.$
In the case $\delta<\mathcal R_M,$ we will write (with a slight abuse of notation):
\[
\rho(\delta):=\lim_{s\rightarrow\delta}\rho(s) \quad\text{and}\quad \rho'(\delta):=\lim_{s\rightarrow\delta}\rho'(s).
\]
(Notice that, from the maximality of $\delta,$ we must have
$\rho(\delta)=1$ or $\rho'(\delta)=0.$)
\end{definition}

\begin{remark} \label{rem-graph}
From Lemma \ref{lem-graph} and equalities \eqref{eq-principalcurvatures}
and \eqref{eq-phi10}, a function $\rho$ satisfying
conditions (C2)--(C3) defines a rotational
$W$-graph $\Sigma$ over the punctured open ball $B_\delta(o)-\{o\}\subset M,$
whose function $\phi,$ up to a constant, is given by
\begin{equation} \label{eq-phi2}
\phi(s)=\int_{0}^{s}\frac{\rho(u)}{\sqrt{1-\rho^2(u)}}du\,, \,\, \,\,\, s\in(0,\delta).
\end{equation}
Also, the condition (C1) and the finiteness of the two first limits in (C4) (together with \eqref{eq-principalcurvatures})
imply that $\Sigma$ extends $C^2$-smoothly to the puncture $o.$ Analogously,
when $\delta<+\infty,$ the finiteness of the last limit in (C4)
gives that $\Sigma$ extends $C^2$-smoothly to its boundary $f_\delta(\s^{n-1})\times\{\phi(\delta)\}.$
\end{remark}

Now, we are in position to state and prove our first main result.

\begin{theorem} \label{th-main}
Let $M^n$ be either a simply connected space form
$\q$ or a rank-one symmetric space of
noncompact type $\hf.$
Given an $M$-admissible  Weingarten function $W,$
let $\rho:(0,\delta)\rightarrow[0,1)$ be
its associated function. Then, the following assertions hold:
\begin{itemize}[parsep=1ex]
\item[\rm i)] If $\delta<\mathcal R_M,$ $\rho(\delta)=1$ and $\rho'(\delta)>0,$
there exists an embedded strictly convex rotational $W$-sphere
in $M\times\R$ which is symmetric with respect to a horizontal hyperplane.

\item[\rm ii)] If $\delta=+\infty$ (so that $M$ is \,$\R^n$ or \,$\hf$), there
exists a rotational strictly convex  entire  $W$-graph in $M\times[0,+\infty)$
which is tangent to $M\times\{0\}$ at a single point, and whose height function is unbounded above.
\end{itemize}
In particular, if $W$ is elliptic, the $W$-hypersurfaces in (i) and (ii) are elliptic.
Consequently, if (ii) occurs  for such a $W,$ there is no compact $W$-hypersurface in \,$M\times\R.$
\end{theorem}

\begin{proof}
 Assume the hypotheses in (i), and let $\Sigma'$ be the rotational $(f_s,\phi)$-graph defined by
 the function $\phi$ in \eqref{eq-phi2}. As we pointed out in
 Remark \ref{rem-graph}, $\Sigma'$ is
 defined over $B_\delta(o)\subset M$ and constitutes a $W$-hypersurface
 of $M\times\R.$ Also, by~\eqref{eq-principalcurvatures},
 $\Sigma'$ is strictly convex.

 Let us show that the function $\phi$ defining $\Sigma'$ is bounded in $(0,\delta).$ Indeed,
 since $\rho'(\delta)>0,$ there exist $a,\delta_0>0,$ $0<\delta-\delta_0<\delta,$ such that
 $\rho'(s)\ge a \,\forall s\in (\delta-\delta_0,\delta).$ Besides,
 $0\le\rho<1$ and $\rho(\delta)=1.$ Hence,
 \begin{eqnarray}
\int_{\delta-\delta_0}^{\delta}\frac{\rho(s)ds}{\sqrt{1-\rho^2(s)}}
& \le & \int_{\delta-\delta_0}^{\delta}\frac{\rho'(s)ds}{\rho'(s)\sqrt{1-\rho^2(s)}}
\le\frac{1}{a} \int_{\rho(\delta-\delta_0)}^{1}\frac{d\rho}{\sqrt{1-\rho^2}} \nonumber\\
 & = & \frac{1}{a}\left(\frac{\pi}{2}-\arcsin(\rho(\delta-\delta_0))\right)\le \frac{\pi}{2a}\,, \nonumber
\end{eqnarray}
which implies that $\phi$ is bounded. So, we can set $\phi(\delta)$ for the limit
of $\phi(s)$ as $s\rightarrow\delta.$

Since $\rho(\delta)=1,$ we have from \eqref{eq-phi2} that $\phi'(s)\rightarrow+\infty$ as
$s\rightarrow\delta.$ This gives that,
along $\partial\Sigma'=f_\delta(\s^{n-1})\times\{\phi(\delta)\},$
the tangent spaces are all parallel to $\partial_t$\,. Therefore,
since $\Sigma'$ extends $C^2$-smoothly to its boundary (see Remark \ref{rem-graph}),
if we denote by $\Sigma''$ the reflection of $\Sigma'$ with respect to $M\times\{\phi(\delta)\},$
we have that
\[
\Sigma:={\rm closure}\, \Sigma'\cup {\rm closure}\, \Sigma''
\]
is a strictly convex rotational $W$-sphere of $M\times\R.$
This proves (i).

Now, let us assume that the hypotheses in (ii) hold.
In this case, since $\delta=+\infty,$
the $(f_s,\phi)$-graph $\Sigma$ defined by $\phi$ in \eqref{eq-phi2}
is an entire rotational $W$-graph of
$M\times\R.$ Also, since $\phi(0)=0$ and
$\phi(s)>0$ for any $s>0,$ $\Sigma$ is contained in the closed half-space
$M\times[0,+\infty),$ being tangent to $M\times\{0\}$ at $o.$

It remains to prove that the height function
of $\Sigma$ is unbounded above. For that, we fix $\delta_0 >0$
and set $\iota_{\delta_0}(s)$ for the infimum
of the function $u\mapsto\rho(u)/\sqrt{1-\rho^2(u)}$ on $[\delta_0, s)$, $s>\delta_0.$
It is easily seen that $\iota_{\delta_0}(s)$
is bounded away from zero. Therefore,
\[
\phi(s)=\int_{0}^{s}\frac{\rho(u)}{\sqrt{1-\rho^2(u)}}du\ge
\int_{\delta_0}^{s}\frac{\rho(u)}{\sqrt{1-\rho^2(u)}}du\ge \iota_{\delta_0}(s)(s-\delta_0) \,\,\,\, \forall s>\delta_0,
\]
which gives that $\phi$ is unbounded above.

The last assertion regarding the non existence of compact elliptic $W$-hypersurfaces in the occurrence
of (ii)  follows  from the
Maximum Principle (see Remark \ref{rem-MP}).
\end{proof}

\begin{figure}[htbp]
\includegraphics{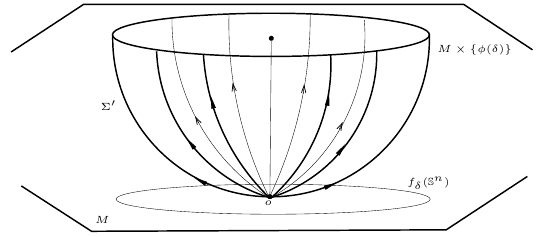}
\caption{\small A $W$-hemisphere in $M\times\R$ and its $T$-trajectories.}
\label{fig-sphere}
\end{figure}

In what follows, for $\epsilon\in\{0,-1,1\},$ we shall consider the trigonometric functions
$\tan_\epsilon=\sin_\epsilon/\cos_\epsilon$ and $\cot_\epsilon=1/\tan_\epsilon,$
where $\cos_\epsilon$ and $\sin_\epsilon$ are defined as in
Table \ref{table-trigfunctions}.
\begin{table}[thb]%
\centering %
\begin{tabular}{cccc}
\toprule %
   {{\small\rm Function}}                & $\epsilon=0$ & $\epsilon=1$   & $\epsilon=-1$ \\\otoprule %
$\cos_\epsilon (s)$    & $1$          & $\cos s$        & $\cosh s$     \\\midrule
$\sin_\epsilon (s)$    & $s$          & $\sin s$         & $\sinh s$   \\\bottomrule
\end{tabular}
\vtt
\caption{Definition of $\cos_\epsilon$ and $\sin_\epsilon$}
\label{table-trigfunctions}
\vspace{-.8cm}
\end{table}

\begin{example}
Given constants $a, b$ and $c$ with $a+b\ne0,$ and $\epsilon=\pm 1,$ consider the following
elliptic general Weingarten function $W_\epsilon\colon\R^2\times[0,1]\rightarrow\R:$
\[
W_\epsilon(k_1,k_2,\theta^2)=k_1k_2-\frac{c-\epsilon a\theta^2}{a+b}\cdot
\]
For $n=2,$ Gauss equation \eqref{eq-barcurvaturetensor} reduces to
$K=H_2+\epsilon\theta^2,$ where $K$ is the intrinsic curvature of $\Sigma$ and $H_2=k_1k_2$ is its
extrinsic  curvature. Hence, a $W_\epsilon$-surface
$\Sigma$ of $\mathbb Q_\epsilon^2\times\R$ satisfies the relation
\[
aK+bH_2=c.
\]
(Such surfaces were studied in \cite{folha-penafiel}.)

Let us see that $W_\epsilon$ is  $\mathbb Q_\epsilon^2$-admissible, provided that
\begin{equation} \label{eq-conditionsexample}
a\epsilon>0, \quad a+b>0, \quad\text{and}\quad c-a\epsilon>0.
\end{equation}
To that end, consider  a family $\mathscr F$ of concentric geodesic circles
in $\mathbb Q_\epsilon^2,$ and recall that, with the outward orientation, a geodesic
circle of radius $s$ in $\mathbb Q_\epsilon^2$ has curvature $k(s)=-\cot_\epsilon(s).$
In this setting, the ODE \eqref{eq-edo} takes the form
\begin{equation} \label{eq-edoexample}
(a+b)\cot_\epsilon(s)\rho(s)\rho'(s)+\epsilon a(1-\rho^2(s))=c.
\end{equation}
Separating variables and integrating, one easily concludes that
the solution $\rho$ of \eqref{eq-edoexample} satisfying $\rho(0)=0$ is given by
\[
\rho(s)=\left(\frac{c-a\epsilon}{a\epsilon}((\cos_\epsilon(s))^{\frac{-2a}{a+b}}-1)\right)^{1/2}.
\]
In particular, $\rho$ is increasing and satisfies $\rho(\delta)=1,$ where
\[
\delta:=\left(\frac{c}{c-a\epsilon}\right)^{\frac{a+b}{-2a}}.
\]
Also, from \eqref{eq-edoexample}, we have that
$$\rho'(\delta)=\frac{c}{a+b}\tan_\epsilon(\delta)>0.$$
In addition, a direct application of  L'Hôpital's rule gives that
\[
\lim_{s\rightarrow 0}(\rho(s)\cot_\epsilon(s))=\left(\frac{c-a\epsilon}{a+b}\right)^{1/2}.
\]
This limit, together with \eqref{eq-edoexample}, yields
\[
\lim_{s\rightarrow 0}\rho'(s)=\frac{c-a\epsilon}{\sqrt{(a+b)(c-a\epsilon)}},
\]
which completes the proof that $W_\epsilon$ is $\mathbb Q_\epsilon^2$-admissible with
associated function $\rho.$

Clearly, $\rho$ satisfies the conditions of Theorem \ref{th-main}-(i). Therefore,
for all constants $a, b,$ and $c$ satisfying \eqref{eq-conditionsexample}, there exists
a rotational elliptic  $W_\epsilon$-sphere in $\mathbb Q_\epsilon^2\times\R.$
\end{example}

\begin{example}
Given $b<0<a$ and an integer $n\ge 2,$ write $\alpha:=\left(-{a}/{b}\right)^{\frac{1}{n-1}}$ and
set
\begin{equation} \label{eq-c}
c=a(n-1)\alpha.
\end{equation}
Under these conditions, the Weingarten function
$W\in C^{\infty}(\R^n)$ given by
\[
W=aH+bH_n-c,
\]
is $\q$-admissible.
Indeed,
considering Lemma \ref{lem-graph}
for $W,$ $M=\q,$  and $\mathscr F$
as in \eqref{eq-isoparametricspheres}, we have
that the ODE \eqref{eq-edo} takes the form
\begin{equation} \label{eq-edoexample01}
a\left((n-1)\cot_\epsilon(s)\rho(s)+\rho'(s)\right)+b((\cot_\epsilon(s)\rho(s))^{n-1}\rho'(s))=c.
\end{equation}

Thus, defining
\[
\delta:=\left\{
\begin{array}{lll}
\arctan_\epsilon({1}/{\alpha}) & \text{if} & \epsilon\ne-1 \,\,\text{or}\,\,\, \epsilon=-1 \,\, \text{and} \,\, \alpha>1,\\[1ex]
+\infty                           & \text{if} & \epsilon=-1 \,\,\, \text{and} \,\,\, \alpha\le 1,
\end{array}
\right.
\]
and considering \eqref{eq-c}, we have that
\[
\rho(s)=\alpha\tan_\epsilon(s), \,\,\, s\in[0,\delta),
\]
is the solution of \eqref{eq-edoexample01} satisfying
$\rho(0)=0.$ Moreover, for $\delta<+\infty,$ we have that
$\rho(s)\rightarrow 1$ as $s\rightarrow\delta.$ It is also clear
that $\rho$ satisfies the conditions (C1)--(C4) of Definition \ref{def-weigartenpair},
which implies that $W$ is $\q$-admissible.

Therefore, by Theorem \ref{th-main}, there exists a rotational strictly convex $W$-hypersurface
$\Sigma$ in
$\q\times\R$ which is a sphere, if $\delta<+\infty,$ or an entire graph, if $\delta=+\infty.$
We remark that the  principal curvatures $k_1,\dots, k_{n-1}$ of $\Sigma$ are all constant and
equal to $\alpha.$ 
For $\epsilon=0,$
$\Sigma$ is the totally geodesic sphere of radius $1/\alpha.$
\end{example}

Let $M$ be either $\hf$ of $\s^n.$ It was proved
in \cite{delima-manfio-santos} that, for all $c>0,$ the
elliptic Weingarten function
$W_c=H_r-c$
is $M$-admissible. Moreover, for $M=\hf,$
there is a constant $C(\mathbb F)>0$ such that
the associated function $\rho_c$ to $W_c$ satisfies the
conditions of Theorem \ref{th-main}-(i) (resp. Theorem \ref{th-main}-(ii))
if $c>C(\mathbb F)$ (resp. $c\le C(\mathbb F)$).
In our next result, we show that this situation is somewhat typical.

\begin{theorem} \label{th-homogeneous}
Let $M$ be as in Theorem \ref{th-main}. Given $c>0,$
let $W_c$ be the Weingarten function defined by
\,$W_c=f-c,$ where $f=f(k_1,\dots, k_n)$  is a
symmetric homogeneous function defined on an open cone $\Gamma\supset\Gamma_+$ in
$\R^n$. Suppose that, for
some $c_0>0,$ $W_{c_0}$ is $M$-admissible
with associated function $\rho_0:[0,\delta_0)\rightarrow[0,1).$
Then, one has:
\begin{itemize}[parsep=1ex]
 \item[\rm i)] If $\delta_0<\mathcal R_M,$ there is a constant $c_1\ge c_0$ with the following
 property: For all $c>c_1,$ $W_c$ is $M$-admissible and
there exists an embedded
 $W_c$-sphere in $M\times\R$ as in the statement of Theorem \ref{th-main}-(i).
 \item[\rm ii)] If $\delta_0=+\infty,$ for all positive $c\le c_0$\,,
$W_c$ is $M$-admissible and
there exists an entire rotational $W_c$-graph in $M\times[0,+\infty)$
as in the statement of Theorem \ref{th-main}-(ii).
\end{itemize}
\end{theorem}

\begin{proof}
First, we present a general construction that will be used along the proof.
From the hypothesis, we have that the function $\rho_0:[0,\delta_0)\rightarrow[0,1)$
satisfies
\begin{equation} \label{eq-c0}
f(-k_1^s\rho_0,\dots,-k_{n-1}^s\rho_0,\rho_0')= c_0\,.
\end{equation}
Given $c> 0,$ multiplying both sides of \eqref{eq-c0} by $c/c_0$\,, and denoting
by $d$ the degree of homogeneity of the Weingarten function $f,$ one gets
\[
f(-(c/c_0)^{1/d}k_1^s\rho_0,\dots,-(c/c_0)^{1/d}k_{n-1}^s\rho_0,(c/c_0)^{1/d}\rho_0')= c\,,
\]
which implies that the function $\ol{\rho}_c\colon [0,\delta_0)\to [0,1)$
defined by
\begin{equation} \label{eq-rhoc}
\ol{\rho}_c(s)=\left(\frac{c}{c_0}\right)^{1/d}\rho_0(s)
\end{equation}
satisfies
$W_c(-k_1^s\ol{\rho}_c(s),\ldots,-k^2_{n-1}\ol{\rho}_c(s),\ol{\rho}_c'(s))
= 0$. In particular, we may possibly
extend $\ol{\rho}_c$ past $\delta_0$ or restrict $\ol{\rho}_c$ to a
subinterval to obtain a solution
$\rho_c\colon[0,\delta_c)\to [0,1)$ to
\begin{equation}\label{eqfc}
f(-k_1^s\rho_c(s),\dots,-k_{n-1}^s\rho_c(s),\rho_c'(s))= c,
\end{equation}
where $\delta_c$ is maximal with respect to conditions (C2) and (C3)
of Definition \ref{def-weigartenpair}.
Concerning the $M$-admissibility of $W_c$,
it is straightforward to see that $\rho_c$ satisfies (C1) and the
first two conditions of (C4), hence $W_c$ will be
$M$-admissible, if $\delta_c = \infty$ or if, when $\delta_c<\infty$,
$\lim_{s\to \delta_c}\rho'_c(s)$ exists and is finite.

Having defined the family $\{\rho_c\}_{c>0}$, we next prove (i), so we assume that $\delta_0<\mathcal R_M.$ In this case, set
\[
c_1:=\frac{c_0}{(\rho_0(\delta_0))^d}\geq c_0.
\]
For a given $c>c_1$, let $\rho_{c}$ be defined as above. We claim that $\delta_c < \delta_0$.
In fact, if $\delta_c\geq \delta_0$,
then the fact that
$\rho_c\vert_{[0,\delta_0)} = \left(\frac{c}{c_0}\right)^{1/d}\rho_0$,
implies that
\[
\rho_c(\delta_0)=\left(\frac{c}{c_0}\right)^{1/d}\rho_0(\delta_0)>\left(\frac{c_1}{c_0}\right)^{1/d}\rho_0(\delta_0)=1,
\]
contradicting the maximality of $\delta_c$ with respect to condition (C2)
of Definition \ref{def-weigartenpair}.
In particular, $\rho_0'(\delta_c)>0$ and we have that
$$\lim_{s\to \delta_c}\rho_{c_1}'(s) =
\left(\frac{c}{c_0}\right)^{1/d}\rho_0'(\delta_c)>0.$$
This proves that $W_{c}$ is is $M$-admissible and that
$\rho_c(\delta_c)=1$.
In particular, Theorem \ref{th-main}-(i) applies to $W_c$, and this
proves (i).

Let us assume now that $\delta_0=+\infty$ to prove (ii). Then, for all positive $c\le c_0$\,, one has
\[
\rho_c(s)=\left(\frac{c}{c_0}\right)^{1/d}\rho_0(s)\le \rho_0(s)<1 \,\,\, \forall s\in [0,+\infty),
\]
which clearly implies that $\rho_c:[0,+\infty)\rightarrow [0,1)$ is well defined and satisfies the conditions (C1)--(C4) of
Definition~\ref{def-weigartenpair}, showing that
$W_c$ is $M$-admissible.
Therefore, Theorem \ref{th-main}-(ii) applies. This shows (ii) and
finishes our proof.
\end{proof}


\begin{remark}
Theorem~\ref{th-homogeneous} can be improved under the additional
assumption that $W_c$ is uniformly elliptic, in the sense that
there exists a constant $a>0$ such that
$\frac{\partial f}{\partial x_i}\geq a$ in $\Gamma^+$.
In this case,  \emph{$W_c$ is $M$-admissible for any $c>0$.}
To see this, we make use of the construction of
$\rho_c\colon[0,\delta_c)\to[0,1)$ as above. As already explained,
to prove that $W_c$ is $M$-admissible it suffices to show that,
if $\delta_c <\infty$,
$\lim_{s\to \delta_c} \rho'_c(s)$ exists and is finite.

First, note that if $\delta_c\leq \delta_0$, then
$\lim_{s\to \delta_c}\rho'_c(s)
= \left({c}/{c_0}\right)^{1/d}\rho'_0(\delta_c)<\infty$,
so we may assume that $\delta_c>\delta_0$.
In this case, set 
$$\lambda_i = -k_i^{\delta_c}\rho_c(\delta_c), \,\,\,i\i=1,2,\ldots,n-1.$$
Then, the assumption that
$\frac{\partial_f}{\partial_{x_n}}\geq a>0$ implies that
$\lim_{x\to \infty} f(\lambda_1,\,\ldots,\,\lambda_{n-1},x) = \infty$,
hence~\eqref{eqfc} implies that
the function $\rho_c'(s)$ is uniformly bounded.

Next, we prove that the limit $\lim_{s\to \delta_c} \rho'_c(s)$ exists.
Let $(s_m)_{m\in \N}$ and $(t_m)_{m\in\N}$ be two sequences in
$(0,\delta_c)$ with $s_m,t_m\nearrow\delta_c$ and such that
$$\lim_{m\to \infty}\rho'_c(s_m) = \alpha_1\in [0,\infty),\quad
\lim_{m\to \infty}\rho'_c(t_m) = \alpha_2\in [0,\infty).$$
Then,~\eqref{eqfc} implies that
$$f(\lambda_1,\ldots,\lambda_{n-1},\alpha_1) =
c=f(\lambda_1,\ldots,\lambda_{n-1},\alpha_2),$$
from where we obtain that $\alpha_1 = \alpha_2$, since
$\frac{\partial f}{\partial x_n}>a>0$.
Thus, $\lim_{s\to \delta_c}\rho'(s)$ exists and is finite,
proving that $W_c$ is $M$-admissible. 
\end{remark}


\begin{example}
Given constants $a, b>0,$ let $f\colon\R^n\rightarrow\R$ be the symmetric homogeneous function
of degree $2$ defined by
\[
f=a\|A\|^2+bH_2.
\]

For any given $c>0,$ $W_c=f-c$ is clearly an elliptic Weingarten function.
Considering Lemma \ref{lem-graph} for $M=\q,$ $W=W_c,$ and $\mathscr F$ as in
\eqref{eq-isoparametricspheres},
the ODE \eqref{eq-edo} takes the form:
\begin{equation}\label{eq-edoalpha}
a(\rho'(s))^2+(n-1)b\cot_\epsilon(s)\rho(s)\rho'(s)+(\alpha\cot_\epsilon(s)\rho^2(s)-c)=0,
\end{equation}
where $\alpha:=(n-1)a+b{{n-1}\choose{2}}.$

Notice that, for any $a>0,$ we can choose  $b>0$ in such a way that
\begin{equation}\label{eq-alpha}
(n-1)^2b^2-4a\alpha=0.
\end{equation}
Indeed, this equality is equivalent to
the quadratic equation for $b$:
\begin{equation}\label{eq-b}
(n-1)b^2-2(n-2)ab-4a^2=0,
\end{equation}
which is easily seem to have a positive root.

Assuming \eqref{eq-alpha}, we can solve \eqref{eq-edoalpha} for $\rho',$ obtaining
\begin{equation}  \label{eq-linearode}
\rho'(s)=-\frac{(n-1)b}{2a}\cot_\epsilon(s)\rho(s)+\frac{\sqrt{ac}}{a}\cdot
\end{equation}
Setting $\beta:=\frac{(n-1)b}{2a}$ and $\mathcal R_\epsilon:=\mathcal R_{\q}$, as in~\eqref{eq-radius},
the standard method of resolution of linear  ODE's gives that
\[
\rho(s)=\frac{\sqrt{ac}}{a}\frac{\int_{0}^{s}\sin_\epsilon^{\beta}(u)du}{\sin_\epsilon^{\beta}(s)}, \,\,\, s\in(0,\mathcal R_\epsilon),
\]
is a (positive) solution to \eqref{eq-linearode}.
A direct computation yields
\begin{equation}  \label{eq-conditions1}
\lim_{s\rightarrow 0}\rho(s)=0 \quad\text{and}\quad \lim_{s\rightarrow 0}(\cot_\epsilon(s)\rho(s))=\frac{\sqrt{ac}}{a(\beta+1)}\cdot
\end{equation}
Also, from \eqref{eq-linearode} and the second equality in \eqref{eq-conditions1}, we have
\begin{equation}  \label{eq-conditions2}
\lim_{s\rightarrow 0}\rho'(s)=\frac{\sqrt{ac}}{a(\beta+1)}>0.
\end{equation}
Hence, $\rho'(s)>0$ for $s>0$ sufficiently small. In fact, one has $\rho'>0$ in
$(0,\mathcal R_{\epsilon}).$ Otherwise, there would exist $s_0>0$ such that $\rho'(s_0)=0$
and $\rho'(s)>0$ for all $s\in(0,s_0).$ Then, \eqref{eq-linearode} would give
\[
\rho''(s_0)=\frac{\beta\rho(s_0)}{\sin_\epsilon^2(s_0)}>0,
\]
i.e, $s_0$ would be a local minimum for $\rho,$ contradicting that
$\rho'>0$ in $s\in(0,s_0).$

Again by a direct computation, we have
\[
\lim_{s\rightarrow\mathcal R_{\epsilon}}\rho(s)=
\left\{
\begin{array}{lcl}
+\infty & \text{if} & \epsilon=0,\\[1ex]
\frac{\sqrt{ac}}{\beta} & \text{if} & \epsilon=-1,\\[1ex]
\sqrt{ac} \,I(\beta) & \text{if} & \epsilon=1,
\end{array}
\right.
\]
where $I(\beta):=\int_{0}^{\pi/2}\sin^\beta(s)ds\le 1.$ Thus, we
can define $\delta_c:=\rho^{-1}(1)<\mathcal R_{\epsilon}$ in any of the following occurrences:
\begin{itemize}[parsep=1ex]
\item $\epsilon=0.$

\item $\epsilon=-1$ and $\frac{\sqrt{ac}}{\beta}>1.$

\item $\epsilon=1$ and $\sqrt{ac}\,I(\beta)>1.$
\end{itemize}
In any of these cases, it follows from \eqref{eq-linearode} that $$\lim_{s\rightarrow\delta_c}\rho'(s)<+\infty.$$
Finally, we set $\delta_c=+\infty$ if $\epsilon=-1$ and $\frac{\sqrt{ac}}{\beta}\le 1.$

It follows from the above considerations that, for any $a>0$ and any $b>0$ satisfying \eqref{eq-b},
$W_c=f-c$ is $\q$-admissible for all $c>0.$ (However, for
$\epsilon=1,$ $\delta_c=\pi/2=\mathcal R_\epsilon$ if $\sqrt{ac}\,I(\beta)\le 1.$)
Furthermore, Theorem \ref{th-main}-(i) applies in the case $\delta_c<\mathcal R_{\epsilon},$ and
Theorem \ref{th-main}-(ii) applies in the case $\epsilon=-1$ and $\delta_c=+\infty.$ (Notice that, for
$\epsilon=0,$ the $W_c$-sphere obtained from Theorem \ref{th-main}-(i) is totally umbilical and has radius
$R={a(1+\beta)}/{\sqrt{ac}}.$)
\end{example}

\section{Symmetric Hypersurfaces of Constant Scalar Curvature in $\q\times\R.$} \label{sec-symmetricCSC}

In this section, we consider hypersurfaces of constant scalar curvature
(CSC) of $\q\times\R,$ $\epsilon\ne 0,$ which are \emph{symmetric}, meaning that they are
invariant by elliptic, parabolic or hyperbolic isometries. Any such
isometry is determined by a parallel family $\mathscr F$ of totally umbilical hypersurfaces of $\q,$
that is, geodesic spheres (elliptic), horospheres (parabolic) or equidistant hypersurfaces (hyperbolic).
In particular, hypersurfaces invariant by elliptic isometries are the rotational ones.

\begin{remark}
Regarding the notation $\q,$ we will assume from now on that  $\epsilon\ne 0,$  that is, $\q$ will refer
only to $\h^n$ or  $\s^n.$
\end{remark}

Let us start by showing that CSC hypersurfaces in $\q\times\R$ are
Weingarten hypersurfaces. Indeed, assuming $n\ge 3,$ consider a hypersurface $\Sigma$ of \,$\q\times\R$
and set $K(X,Y)$ for its sectional curvature determined by $X,Y\in T\Sigma.$
Given an orthonormal frame $\{X_1\,, \dots , X_n\}$ of principal directions
in $T\Sigma,$ it follows from Gauss equation \eqref{eq-gauss001} that:
\begin{equation} \label{eq-gauss}
K(X_i\,,X_j)=k_ik_j+\epsilon(1-\|T_{ij}\|^2), \,\, \, i\ne j\in\{1,\dots ,n\},
\end{equation}
where $k_i$ is the principal curvature in the direction $X_i$, and
$T_{ij}$ is the orthogonal projection of $T$ on the plane of $T\Sigma$ determined
by $X_i$ and $X_j$\,.

In this setting,
we have that the (non normalized) \emph{scalar curvature} $S$ of $\Sigma$ is:
\[
S=\sum_{i\ne j}K(X_i,X_j).
\]
Considering the equality $T=\partial_t-\Theta N$ and noticing that
\[
\sum_{i\ne j}\|T_{ij}\|^2=2(n-1)\|T\|^2=2(n-1)(1-\theta^2),
\]
we have from equation \eqref{eq-gauss} that
\begin{equation} \label{eq-scalar0}
S=2H_2+\epsilon (n-1)(2\theta^2+n-2).
\end{equation}

Therefore, given $c\in\R,$ defining $W_c:\R^n\times[0,1]\rightarrow\R$ by
\begin{equation} \label{eq-scalar}
W_{c}(k_1\,, \dots ,k_n,\theta^2)=2H_2(k_1\,, \dots ,k_n)+\epsilon (n-1)(2\theta^2+n-2)-c,
\end{equation}
we have that $W_c$ is elliptic Weingarten, and also that
a $W_c$-hypersurface of $\q\times\R$ has constant scalar curvature $S=c.$

Now, choose  a family
\[\mathscr F:=\{f_s:M_0\rightarrow \q\,;\, s\in I \}\]
of parallel  totally umbilical hypersurfaces of $\q,$
and write  $\alpha(s)$
for the principal curvature of $f_s\in\mathscr F.$
In this setting, the equalities \eqref{eq-principalcurvatures}
take the form:
\[
k_i=-\alpha\rho \,\,\,(i=1,\dots,n-1) \quad\text{and}\quad k_n=\rho',
\]
which gives
\[
2H_2(k_1,\dots, k_n)=-2(n-1)\alpha\rho\rho'+(n-1)(n-2)\rho^2\alpha^2.
\]
Since $\theta^2=1-\rho^2,$
we have that the equality $W_c(k_1,\dots,k_n)=0$ for
$k_1,\dots, k_n$ as above (that is,
the ODE \eqref{eq-edo} for  $W=W_c$)
is equivalent to
\begin{equation}\label{eq-edo3}
(n-1)(-2\alpha\rho\rho'+((n-2)\alpha^2-2\epsilon)\rho^2+n\epsilon)=c.
\end{equation}

Setting $\tau:=\rho^2,$ equation \eqref{eq-edo3} becomes
\begin{equation}\label{eq-edo4}
(n-1)(-\alpha\tau'+((n-2)\alpha^2-2\epsilon)\tau+n\epsilon)=c.
\end{equation}

Equation \eqref{eq-gauss} also gives that the sectional curvatures $K(X_i\,, X_j)$ of
an $(f_s,\phi)$-graph $\Sigma$ as in Lemma \ref{lem-graph} are given by (recall that $T$ is a principal direction of $\Sigma$):
\begin{equation}\label{eq-scalargraph}
 \begin{aligned}
 K(X_i\,,X_j) &=\alpha^2\rho^2+\epsilon \quad (i\ne j=1,\dots, n-1). \\
 K(X_i,X_n) &=-\alpha\rho\rho'+\epsilon(1-\rho^2) \quad (i=1,\dots, n-1). \\
 \end{aligned}
\end{equation}
In particular, $\Sigma$ has constant sectional curvature if and only if
$\rho$ satisfies
\begin{equation}\label{eq-edo10}
\alpha\rho\rho'+(\alpha^2+\epsilon)\rho^2=0.
\end{equation}

\begin{remark}
Throughout this section, $W_c$ will always denote
the Weingarten function defined in \eqref{eq-scalar}. We stress that
$\Sigma\subset\q\times\R$ has constant scalar curvature $c$ if and only if
$\Sigma$ is a $W_c$-hypersurface.
\end{remark}

\subsection{Rotational CSC hypersurfaces of $\q\times\R$}
Setting $\mathcal R_\epsilon:=\mathcal R_{\q}$, as in~\eqref{eq-radius},
we have that
any principal curvature of a geodesic sphere of $\q$ of radius
$s\in (0,\mathcal R_\epsilon)$ is
$\alpha(s)=-\cot_\epsilon s.$
In this case, equation \eqref{eq-edo4} takes the form
\begin{equation}\label{eq-tauode}
\tau'(s)=a(s)\tau(s)+b(s), \,\,\, s\in (0,\mathcal R_\epsilon),
\end{equation}
where the functions $a$ and $b$ are given by
\begin{equation}\label{eq-a&b}
a=-(n-2)\cot_\epsilon+2\epsilon\tan_\epsilon \qquad\text{and}\qquad b=\left(\frac{c}{n-1}-\epsilon n\right)\tan_\epsilon.
\end{equation}

The general solution
to \eqref{eq-tauode} is as follows. For fixed $s_0\in (0,\,+\infty)$ and
$\tau_0\in \R$,
\begin{equation} \label{eq-generalsolution}
\tau(s) =\frac{1}{\mu(s)}\left(\tau_0+\int_{s_0}^{s}{b(u)}{\mu(u)}du\right),
\end{equation}
where  $\mu(s)=\exp\left(-\int_{s_0}^sa(u)du\right)$ and $s\in (0,\,+\infty)$.
A direct computation gives
\begin{equation} \label{eq-tauu1}
\tau(s)=\mathfrak C_n\frac{\sin_\epsilon^n (s)-\sin_\epsilon^n(s_0)}{\sin_\epsilon^{n-2}(s)\cos_\epsilon^2(s)}+
\tau_0\frac{\sin_\epsilon^{n-2}(s_0)\cos_\epsilon^2(s_0)}{\sin_\epsilon^{n-2}(s)\cos_\epsilon^2(s)},
\end{equation}
where $\mathfrak C_n$ is the constant defined as
\begin{equation} \label{eq-cn}
\mathfrak{C}_n:={\frac{c-\epsilon n(n-1)}{n(n-1)}}\,\cdot
\end{equation}

Notice that there exists a solution $\tau$ which is defined at the singular
point $s_0=0$ and satisfies $\tau(0)=0.$ Namely,
\begin{equation} \label{eq-tauu}
\tau(s)=\mathfrak C_n\tan_\epsilon^2(s).
\end{equation}
In this case,
since $\tau$ is required to be a nonnegative function,
we must have $\mathfrak C_n\ge 0,$ that is
$c\ge\epsilon n(n-1).$ If $c=\epsilon n(n-1),$ the function
$\tau$, and so $\rho$, vanishes identically, and the corresponding
$(f_s,\phi)$-graph is nothing but a totally geodesic horizontal
hyperplane of $\q\times\R.$ So, we shall assume $\mathfrak C_n>0.$

If   $c>0$ and $\epsilon=-1,$ we have that $\mathfrak C_n>1,$ so that
$\arctan_\epsilon(1/\sqrt{\mathfrak{C}_n})$ is well defined for these values
of $c$ and $\epsilon$
(as well as for $c>0$ and $\epsilon=1$). On the other hand,
$\mathfrak C_n\le 1$ if $c\le 0$ and $\epsilon=-1.$
These facts and equality \eqref{eq-tauu} imply that, in any case,
$W_c$ is  $\q$-admissible and its
associated function $\rho$ is
\begin{equation} \label{eq-rhoconstantseccurvature}
\rho(s)=\sqrt{\mathfrak{C}_n}\tan_\epsilon(s), \,\, s\in[0,\delta),
\end{equation}
where $\delta>0$ is given by
\[
\delta:=\left\{
\begin{array}{lcl}
\arctan_\epsilon(1/\sqrt{\mathfrak{C}_n}) & \text{if} & c>0\\[1ex]
+\infty & \text{if} & c\le 0.
\end{array}
\right.
\]
(Notice that the condition (C4) in Definition \ref{def-weigartenpair} is easily checked.)
In particular, $\rho$ fulfills the conditions of
Theorem \ref{th-main}, where case (i) occurs if $c>0,$ and
case (ii) occurs if $c\le 0.$ Finally, a direct computation gives that
$\rho$ satisfies \eqref{eq-edo10}.

Summarizing, we have the following result.

\begin{theorem} \label{th-scalar}
Given $n\ge 3,$ for all $c>\epsilon n(n-1),$ there exists a
properly embedded strictly convex (and so elliptic) rotational
$W_c$-hypersurface $\Sigma$ in $\q\times\R$
which is necessarily of constant sectional curvature $K=c/(n(n-1))>\epsilon.$
Furthermore,
if $c>0,$ $\Sigma$ is a sphere as in Theorem \ref{th-main}-(i), and if
$c\le 0,$ $\Sigma$ is an entire graph as in Theorem \ref{th-main}-(ii).
\end{theorem}

Let us consider equation \eqref{eq-tauode} again and look for solutions satisfying
$\tau(\lambda)=1$ and $\tau'(\lambda)<0$ for suitable values of $s_0=\lambda$\,.
In this case, we must have $a(\lambda)+b(\lambda)<0,$ which yields the inequality
\begin{equation} \label{eq-delta}
(n\mathfrak C_n+2\epsilon)\tan_\epsilon^2(\lambda)<(n-2).
\end{equation}

Assume $c>\epsilon n(n-1)$ and define $\delta_\epsilon(c)\in(0,\,\infty]$ as
$$\delta_\epsilon(c) =
{\rm sup}\{\lambda>0\mid
(n\mathfrak C_n+2\epsilon)\tan_\epsilon^2(\lambda)<(n-2)\}.$$
Then, we have:

\begin{theorem} \label{th-CSCannuli}
Given  $c>\epsilon n(n-1)$, there exists a one-parameter family
\[
\mathscr S\:=\{\Sigma(\lambda)\,;\,\lambda\in(0,\delta_\epsilon(c))\}
\]
of properly embedded rotational $W_c$-hypersurfaces
in $\q\times\R$ which are all homeomorphic to the $n$-annulus $\s^{n-1}\times\R.$
In addition, the following assertions hold:
\begin{itemize}[parsep=1ex]
\item[\rm i)] If either $\epsilon=-1$ and $\mathfrak C_n> 1,$ or $\epsilon=1,$
each $\Sigma(\lambda)\in\mathscr S$ is Delaunay-type, i.e., it
is periodic in the vertical direction, and has unduloids as its $T$-trajectories.

\item[\rm ii)] If $\epsilon=-1$ and $\mathfrak C_n\le 1$,
each hypersurface $\Sigma(\lambda)\in\mathscr S$ is symmetric with respect to $\q\times\{0\}$
and has unbounded height function.
\end{itemize}
\end{theorem}

\begin{proof}
Given $s_0=\lambda\in(0,\delta_\epsilon(c)),$ let $\tau$ be the solution
of \eqref{eq-tauode} satisfying $\tau(\lambda)=1.$ Then, from
the definition of $\delta_\epsilon(c)$, and equalities \eqref{eq-tauode} and \eqref{eq-a&b},
it follows that
$\tau'(\lambda)<0,$ so that $\tau$ is decreasing near $\lambda.$ Since
we are assuming $c>\epsilon n(n-1)$, we have that $\mathfrak C_n$ is positive, and
so is $b=n\mathfrak C_n\tan_\epsilon$\,. This, together with \eqref{eq-generalsolution},
implies that $\tau$ is positive on $(\lambda,\mathcal R_\epsilon).$
Also, from \eqref{eq-tauu}, one has
 \begin{equation} \label{eq-limittau01}
 \lim_{s\rightarrow\mathcal R_\epsilon}\tau(s)=\left\{
 \begin{array}{ccl}
 \mathfrak C_n & {\rm if} & \epsilon=-1 \\
 +\infty & {\rm if} & \epsilon=1.
 \end{array}
 \right.
 \end{equation}

Therefore, if $\epsilon=1$ or if $\epsilon=-1$ and $\mathfrak C_n>1,$
there exists
$\bar\lambda>\lambda$ such that
 \[
 0<\tau|_{(\lambda,\bar\lambda)}<1\quad\text{and}\quad \tau(\lambda)=\tau(\bar\lambda)=1.
 \]
 Let us see that $\tau'(\bar\lambda)\ne 0.$ Assuming otherwise, we have
 $a(\bar\lambda)=-b(\bar\lambda),$ so that
 $(n\mathfrak C_n+2\epsilon)\tan_\epsilon^2(\bar\lambda)=n-2.$ In particular,
 $n\mathfrak C_n+2\epsilon>0.$ However,
 \[
 \tau''(\bar\lambda)=a'(\bar\lambda)+b'(\bar\lambda)=\frac{n-2}{\sin_\epsilon^2(\bar\lambda)}+\frac{n\mathfrak C_n+2\epsilon}{\cos_\epsilon^2(\bar\lambda)}>0,
 \]
 which implies that $\bar\lambda$ is a local minimum of $\tau.$ This is a contradiction, since $\bar\lambda$ is
 a maximum for $\tau$ in $(\lambda,\bar\lambda].$ Thus, $\tau'(\bar\lambda)>0$.

 It follows from the above considerations that, setting $\tau_\lambda:=\tau|_{(\lambda,\bar\lambda)}$,
 we can proceed as in the proof of
 Theorem \ref{th-main} and conclude that
the $(f_s,\phi)$-graph $\Sigma'(\lambda)$ with $\rho$-function
$\rho=\sqrt{\tau_\lambda}$ is a bounded
$W_c$-hypersurface of $\q\times\R.$ Moreover, $\Sigma'(\lambda)$
is homeomorphic to
$\s^{n-1}\times (\lambda,\bar\lambda)$ and has boundary
\[
\partial\Sigma'(\lambda)=(S_\lambda(o)\times \{0\})\cup (S_{\bar\lambda}(o)\times \{\phi(\bar\lambda)\}),
\]
where $S_s(o)$ denotes the sphere of $\q$ with radius $s$ and center at $o$ (Fig. \ref{fig-W-annuli}-a).

Since $\rho(\lambda)=\rho(\bar\lambda)=1,$
the tangent spaces of $\Sigma'(\lambda)$ are vertical
along its boundary $\partial\Sigma'(\lambda).$
Moreover, $\Sigma'(\lambda)$ extends $C^2$-smoothly
to $\partial\Sigma'(\lambda),$
for $\tau'(\lambda)$ and $\tau'(\bar\lambda)$ (and so  $\rho'(\lambda)$ and $\rho'(\bar\lambda)$)
are both finite.
Therefore, we obtain a properly embedded rotational $W_c$-hypersurface
$\Sigma(\lambda)$ from $\Sigma'(\lambda)$ by
continuously reflecting it with respect to the horizontal hyperplanes
$\q\times\{k\phi(\bar\lambda)\}, \, k\in\Z.$ This proves (i).

To prove (ii), let us suppose that $\epsilon=-1$ and $\mathfrak C_n\le 1.$
In this case, \eqref{eq-limittau01} yields
\[
0<\tau|_{(\lambda,+\infty)}<1,
\]
so that the $(f_s,\phi)$-graph $\Sigma'(\lambda)$
determined by $\rho=\sqrt{\tau|_{(\lambda,+\infty)}}$ is a $W_c$-hyper\-surface
of $\h^n\times\R$ with boundary
$\partial\Sigma'(\lambda)=S_\lambda(o)\times\{0\}$ (Fig. \ref{fig-W-annuli}-b).
By reflecting $\Sigma'(\lambda)$ with respect to
$\q\times\{0\},$ as we did before, we obtain the embedded $W_c$-hypersurface
$\Sigma(\lambda)$ as stated.

It remains to show that the height function of $\Sigma(\lambda)$ is unbounded.
For that, we have just to observe that the infimum of $\tau$ in $[\lambda,+\infty)$
is positive, since $\tau$ itself is positive in this interval, and its limit as
$s\rightarrow+\infty$ is $\mathfrak C_n.$ So, the same is true for $\rho=\sqrt\tau.$
Therefore,
\[
\phi(s)=\int_{\lambda}^{s}\frac{\rho(u)}{\sqrt{1-\rho^2(u)}}du>
\int_{\lambda}^{s}\rho(u)du > \inf\rho|_{[\lambda,+\infty)}(s-\lambda),
\]
from which we conclude that $\phi$ is unbounded. This finishes the proof of (ii).
\end{proof}

\begin{figure}
 \centering
 \includegraphics[width=6.1cm]{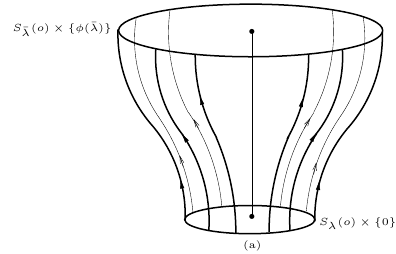}
 \hfill
 \includegraphics[width=4.4cm]{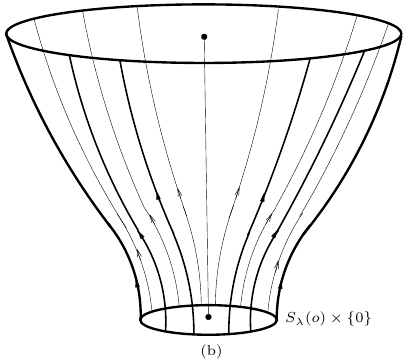}
 \caption{\small The two types of rotational $W_c$-annuli of $\q\times\R$, where the one on the
 right occurs only for $\epsilon=-1.$}
 \label{fig-W-annuli}
\end{figure}

\subsection{Translational CSC Hypersurfaces of $\h^n\times\R$}
Now, we consider hypersurfaces of constant negative scalar curvature
in $\h^n\times\R$ which are invariant by either parabolic or hyperbolic translations.
Recall that, in hyperbolic space $\h^n,$
parabolic translations are isometries which fix parallel families of horospheres,
whereas hyperbolic translations are those which fix parallel families of equidistant hypersurfaces
to a totally geodesic hyperplane of $\h^n.$
In the upper half-space model of $\h^n,$ horizontal Euclidean translations along fixed horizontal directions
are parabolic, and Euclidean homotheties
from the origin are hyperbolic.
It is easily seen that any of these isometries
extend to an isometry of $\h^n\times\R$ which fixes the factor $\R$ pointwise.

Starting with the parabolic case,
let us first consider $(f_s,\phi)$-graphs such that
\[
\mathscr F:=\{f_s:\R^{n-1}\rightarrow\h^n\,;\, s\in(-\infty,+\infty)\}
\]
is a parallel family of horospheres of $\h^n.$
Since the principal curvatures of any horosphere $\mathscr H_s:=f_s(\R^{n-1})$ are all equal to $1$,
equation \eqref{eq-edo4} becomes
\begin{equation} \label{eq-odetau007}
\tau'(s)=n\tau(s)-\frac{c+n(n-1)}{n-1}\,\cdot
\end{equation}

Consider a constant $c\in[-n(n-1),0)$ and write
\begin{equation} \label{eq-bc}
b_c=-\frac{c+n(n-1)}{n-1},
\end{equation}
so that
$0\le -b_c/n<1.$ In this setting, the constant function
\[
\tau_c(s)=-\frac{b_c}{n}\,, \,\,\, s\in(-\infty,+\infty),
\]
is a trivial solution of \eqref{eq-odetau007} satisfying
$0\le \tau_c<1.$ The function $\rho_c=\sqrt\tau_c$ is also constant, and
so it is a solution of \eqref{eq-edo10} (for $\alpha=1$ and $\epsilon=-1$).
Hence, defining
\[
\phi_c(s)=\int_{0}^{s}\frac{\rho_c(u)}{\sqrt{1-\rho_c^2(u)}}du=\frac{\rho_c}{\sqrt{1-\rho_c^2}}s, \,\,\, s\in(-\infty,+\infty),
\]
we have from Lemma \ref{lem-graph} that the $(f_s,\phi_c)$-graph $\Sigma_1(c)$
is entire an has constant sectional curvature $K=\rho_c^2-1=-\theta.$ In particular,
$\Sigma_1(c)$ has constant angle function, and
constant scalar curvature $c=-n(n-1)\theta.$
Also, from identities
\eqref{eq-principalcurvatures}, all principal curvatures of $\Sigma_1(c)$ are
non positive, so that $\Sigma_1(c)$ is convex (but not strictly convex,
since $k_n=\rho_c'=0$).

We add that, for $\rho_c=0,$
$\Sigma_1(c)$ is a horizontal hyperplane of $\q\times\R$ of constant scalar
curvature $c=-n(n-1).$ Moreover,
$\tau_c=-b_c/n\rightarrow 1$ as $c\rightarrow 0,$ which implies that
the angle function of $\Sigma_1(c)$ goes to $0$ as $c\rightarrow 0.$
Consequently, as $c\rightarrow 0,$ $\Sigma_1(c)$ converges
uniformly (on compact sets) to the  cylinder
$\mathscr H_0\times\R$ over the horosphere $\mathscr H_0=f_0(\R^{n-1})\subset\h^n.$
(Notice that, for all $c\in[-n(n-1),0),$ $\Sigma_1(c)\cap(\h^n\times\{0\})=\mathscr H_0.$)

A direct computation also gives that the nonconstant  function
\begin{equation} \label{eq-tau007}
\tau_c(s)=\left(1+\frac{b_c}{n} \right)e^{ns}-\frac{b_c}{n}\,, \,\,\, s\in (-\infty,0],
\end{equation}
with $b_c$ as in \eqref{eq-bc},
is the solution of \eqref{eq-odetau007} which satisfies:
\[
0<\tau_c(s)\le 1=\tau_c(0) \,\,\, \forall s\in (-\infty,0].
\]
Moreover, since $\tau_c'(0)>0,$ as in the proof of Theorem \ref{th-main}, we have that
\[
\phi_c(s)=\int_{0}^{s}\frac{\rho_c(u)}{\sqrt{1-\rho_c^2(u)}}du, \,\,\, s\in (-\infty,0],
\]
is well defined.

Assume $c>-n(n-1).$
Considering equality \eqref{eq-tau007}, one has
$$\rho_c:=\sqrt{\tau}\ge \sqrt{-b_c/n}>0.$$ Thus, for all $s\in (-\infty,0),$
\[
-\phi_c(s)=\int_{s}^{0}\frac{\rho_c(u)}{\sqrt{1-\rho_c^2(u)}}du\ge \int_{s}^{0}\rho_c(u)du\ge\sqrt{-b_c/n}(-s),
\]
which implies that $\phi_c$ is unbounded below. On the other hand, if $c=-n(n-1),$ then $b_c=0,$ which
yields $\rho_c(s)=e^{ns/2}.$ Hence, in this case,
\[
-\phi_c(s)=\int_{s}^{0}\frac{e^{nu/2}}{\sqrt{1-e^{nu}}}du= \frac{2}{n}\int_{e^{ns/2}}^{1}\frac{d\rho_c}{\sqrt{1-\rho_c^2}}=
\frac{2}{n}\left(\frac{\pi}{2}-\arcsin(e^{ns/2})\right),
\]
which gives that $\phi_c(s)>-\pi/n\, \forall s\in (-\infty,0],$ and that
$\phi_c(s)\rightarrow -\pi/n$ as $s\rightarrow-\infty.$

Therefore, the $(f_s,\phi_c)$-graph $\Sigma_2'(c)$ is a
$W_c$-hypersurface of $\h^n\times\R$ with boundary $\mathscr H_0\times\{0\}.$
Also, the height function of $\Sigma_2'(c)$ is unbounded below if $c>-n(n-1)$ and, if $c=-n(n-1),$
$\Sigma_2'(c)$ is contained in the slab $\h^n\times (-\pi/2, 0],$ being
asymptotic to the horizontal hyperplane
$P_{-\pi/2}:=\h^n\times\{-\pi/2\}.$
Notice that $\Sigma_2'(c)$ is nowhere convex, for its
principal curvatures are all negative, except for $k_n=\rho'>0.$
In addition, the tangent
spaces of $\Sigma_2'(c)$ along its boundary are all vertical, for $\phi_c'(s)\rightarrow +\infty$ as
$s\rightarrow 0,$ and $\Sigma_2'(c)$ extends $C^2$-smoothly to $\partial\Sigma_2'(c),$ for
$\rho_c'(0)>0.$

We conclude from the above considerations that the hypersurface $\Sigma_2(c)$ obtained by the union of the closure of
$\Sigma_2'(c)$ with its reflection with respect to the hyperplane $P_0:=\h^n\times\{0\}$ is a
properly embedded hypersurface of $\h^n\times\R$ which is invariant by parabolic translations, since
the vertical projections of its horizontal sections over $\h^n\times\{0\}$ are horospheres
all centered at the same point at infinity (Fig. \ref{fig-W-horographs}).

Finally, we observe that, since $b_c/n\rightarrow -1$ as $c\rightarrow 0,$ given
$\epsilon_0>0,$ there exists $c_0>0$ such that
$$\left|\frac{b_c}{n}+1\right|<\frac{\epsilon_0}{2}\,\,\, \forall c\in (0,c_0).$$
Hence, from \eqref{eq-tau007}, one has
\[
|\tau_c(s)-1|\le
\left|\frac{b_c}{n}+1\right|e^{ns}+\left|\frac{b_c}{n}+1\right|
<\epsilon_0 \,\, \forall s\in(-\infty ,0), \, \forall c\in(0,c_0),
\]
which implies that, as $c\rightarrow 0,$
$\tau_c$ converges uniformly  to the constant function $\tau=1$
on $(-\infty, 0].$ Since
$\partial\Sigma_2'(c)=\mathscr H_0\,\forall c\in [-n(n-1),0),$
likewise the hypersurfaces $\Sigma_1(c)$ above,
$\Sigma_2(c)$ converges
uniformly (on compact sets) to the  cylinder
$\mathscr H_0\times\R$ over the horosphere $\mathscr H_0=f_0(\R^{n-1})\subset\h^n$
as $c\rightarrow 0.$ Notice that $T=\partial_t$ is a principal direction
of $\mathscr H_0\times\R$ whose corresponding  principal curvature vanishes identically.
Besides, its other principal curvatures are all equal to $1,$ and the corresponding principal
directions are all tangent to $\mathscr H_0.$ Thus,
by Gauss equation \eqref{eq-gauss}, $\mathscr H_0\times\R$ is flat, i.e., has
vanishing sectional curvature everywhere.

Therefore, we have the following result.

\begin{theorem} \label{th-parabolicscalar}
Given $n\ge 3$ and $c\in[-n(n-1),0),$ there are two
properly embedded
$W_c$-hypersurfaces $\Sigma_1(c)$ and $\Sigma_2(c)$ in $\h^n\times\R$
which are homeomorphic to $\R^n$ and invariant by parabolic translations. In addition, they have
the following properties:
\begin{itemize}[parsep=1ex]
 \item[\rm i)] $\Sigma_1(c)$ is a convex (nowhere strictly convex) entire graph over $\h^n$ with constant sectional curvature
 $K=c/(n(n-1))\in[-1,0)$ and constant angle function. For $c=-n(n-1),$
 $\Sigma_1(c)$ is a totally geodesic horizontal hyperplane of $\h^n\times\R$ of constant sectional curvature $K=-1.$
 \item[\rm ii)] $\Sigma_2(c)$ is nowhere convex and symmetric with respect to $\h^n\times\{0\}$. If
 $c>-n(n-1),$ the height function of $\Sigma_2(c)$ is unbounded, and
 if $c=-n(n-1),$ $\Sigma_2(c)$ is contained
 in the slab $\h^n\times(-\pi/n,\pi/n)$, being asymptotic to the horizontal hyperplanes
 $P_{-\pi/n}$ and $P_{\pi/n}$.
\end{itemize}
Furthermore, as $c\rightarrow 0,$ both $\Sigma_1(c)$ and $\Sigma_2(c)$ converge uniformly
(on compact sets) to a flat cylinder $\mathscr H_0\times\R$ over a horosphere $\mathscr H_0$ of \,$\h^n.$
\end{theorem}

\begin{figure}
 \centering
 \includegraphics[width=6.1cm]{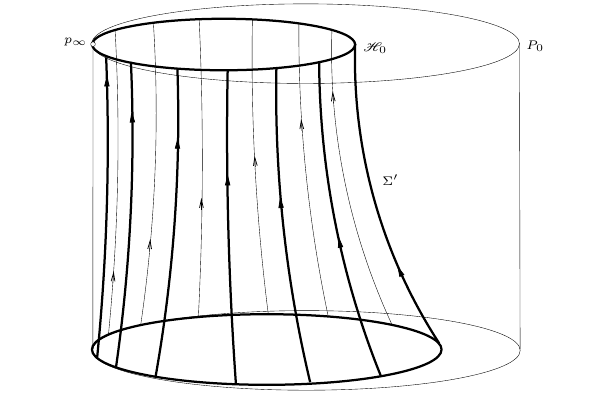}
 \hfill
 \includegraphics[width=6.2cm]{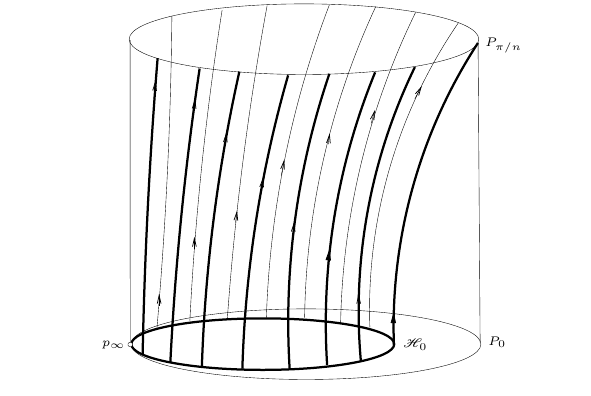}
 \caption{\small Half of the CSC Weingarten hypersurfaces of Theorem \ref{th-parabolicscalar}-(ii),
 where the one on the left is supposed to be unbounded below. (Notice we are considering the Poincaré ball
 model of $\h^n.$)}
 \label{fig-W-horographs}
\end{figure}

Now, we proceed to construct properly embedded $W_c$-hypersurfaces in $\h^n\times\R$ which are invariant
by hyperbolic translations. On that account,
consider an isometric immersion $f_0:\R^{n-1}\rightarrow\h^n$ such that
$f_0(\R^{n-1})$ is a totally geodesic hyperplane of $\h^n$ and let
\[
\mathscr F:=\{f_s:\R^{n-1}\rightarrow\h^n\,;\, s\in (-\infty,+\infty)\}
\]
be the parallel family of equidistant hypersurfaces to $f_0$
in $\h^n.$ Each $f_s$ is totally umbilical with
principal curvatures $\alpha(s)=-\tanh(s).$ Hence, in this setting, and
for $s>0,$ equation \eqref{eq-edo4} takes the form
\begin{equation}\label{eq-tauodehyperbolic}
\tau'(s)=a(s)\tau(s)+b(s), \,\,\, s\in (0,+\infty),
\end{equation}
where the functions $a$ and $b$ are given by
\begin{equation}\label{eq-a&bhyperbolic}
a(s)=-2\coth(s)-(n-2)\tanh(s) \quad\text{and}\quad b(s)=n\mathfrak C_n\coth(s),
\end{equation}
being $\mathfrak C_n$ the constant defined in \eqref{eq-cn} for $\epsilon=-1,$ i.e.,
\[
\mathfrak C_n=\frac{c+n(n-1)}{n(n-1)}\cdot
\]

As for the construction of the hypersurfaces in Theorem \ref{th-CSCannuli}, we look at
solutions $\tau$ of \eqref{eq-tauodehyperbolic} satisfying $\tau(\lambda)=1$ and $\tau'(\lambda)=a(\lambda)+b(\lambda)<0$ for
suitable values of $s_0=\lambda.$ This last inequality is equivalent to
\begin{equation} \label{eq-inequalityhyperbolic}
(n\mathfrak C_n-2)\coth^2(\lambda)<n-2,
\end{equation}
which is valid for all $\lambda>0$ if $n\mathfrak C_n\le 2.$ Otherwise, assuming
$-n(n-1)\le c<0,$ we have that $(n-2)/(n\mathfrak C_n-2)>1.$
Therefore, setting
$$\lambda_0=\lambda_0(c)={\rm arctanh}\sqrt{\mathfrak C_n},$$
and defining the interval $I_c$ as
\[
I_c:=\left\{
\begin{array}{lcl}
(0,+\infty)         & \text{if} & n\mathfrak C_n\le 2\\[1ex]
[\lambda_0,+\infty) & \text{if} & n\mathfrak C_n > 2,
\end{array}
\right.
\]
we have that the inequality
\eqref{eq-inequalityhyperbolic} holds for any
$\lambda\in I_c,$ so that $\tau'(\lambda)<0$
if $\tau(\lambda)=1.$

With this notation, we have the following result.

\begin{theorem} \label{th-hyperbolicscalar}
For any $c\in[-n(n-1),0)$, there exists a one-parameter family
\[
\mathscr S\:=\{\Sigma(\lambda)\,;\,\lambda\in I_c\}
\]
of properly embedded nowhere convex $W_c$-hypersurfaces
in $\h^n\times\R$ which are homeomorphic
to $\R^n$ and invariant by hyperbolic translations. Each
$\Sigma(\lambda)\in\mathscr S$ is
symmetric with respect to $\h^n\times\{0\}$ and has the
following additional properties:
\begin{itemize}[parsep=1ex]
 \item[\rm i)] If $c>-n(n-1),$ the height function of $\Sigma$ is unbounded above and below,
 and $\Sigma(\lambda_0)$ has constant sectional curvature $K=\frac{c}{n(n-1)}\in(-1,0).$

 \item[\rm ii)] If $c=-n(n-1),$ $\Sigma$ is contained
in a slab $\h^n\times(-\sigma\pi/n,\sigma\pi/n),$ $\sigma\in(0,1],$
and is asymptotic to both horizontal hyperplanes
$P_{-\sigma\pi/n}$ and $P_{\sigma\pi/n}$\,.
\end{itemize}
\end{theorem}
\begin{proof}
Given $\lambda\in I_c,$ we have from
\eqref{eq-generalsolution} that the solution $\tau_\lambda$
of \eqref{eq-tauodehyperbolic} satisfying $\tau_\lambda(\lambda)=1$ is given by
\begin{equation}\label{eq-solutionedohyperbolic}
\tau_\lambda(s)=\mathfrak C_n\frac{\cosh^n(s)-\cosh^n(\lambda)}{\cosh^{n-2}(s)\sinh^2(s)}+\frac{\cosh^{n-2}(\lambda)\sinh^2(\lambda)}{\cosh^{n-2}(s)\sinh^2(s)}\,,
\quad s\in [\lambda, +\infty).
\end{equation}

By the definition of $I_c,$ $\tau_\lambda$ is decreasing near $\lambda.$
Also, from \eqref{eq-solutionedohyperbolic}, $\tau_\lambda$ is positive in $[\lambda,+\infty)$.
Let us see that $\tau_\lambda$ has no critical points in this interval. Assuming otherwise,
let $s_1>\lambda$ be such that $\tau_\lambda'(s_1)=0.$ Then, from \eqref{eq-tauodehyperbolic},
we have $\tau_\lambda(s_1)=-b(s_1)/a(s_1),$ which gives
\[
2\tau_\lambda(s_1)-n\mathfrak C_n=n\mathfrak C_n\left(\frac{2\coth(s_1)}{2\coth(s_1)+(n-2)\coth(s_1)}-1 \right)<0.
\]
Therefore,
\[
\tau_\lambda''(s_1)=a'(s_1)\tau_\lambda(s_1)+b'(s_1)=\frac{2\tau_\lambda(s_1)-n\mathfrak C_n}{\sinh^2(s)}-\frac{n-2}{\cosh^2(s)}<0,
\]
which implies that any critical point of $\tau_\lambda$ in $[\lambda,+\infty)$ is a maximum. However, since
$\tau_\lambda$ is decreasing near $\lambda,$ a local maximum point of $\tau_\lambda$ should be preceded by a minimum point.
Thus, $\tau_\lambda$ has no critical points, so that it is decreasing in $[\lambda,+\infty).$

Therefore, the function $\rho_\lambda=\sqrt{\tau_\lambda}$ is well defined in $[\lambda,+\infty)$ and satisfies:
\[
0<\rho_\lambda(s)\le 1=\rho_\lambda(\lambda) \quad\text{and}\quad \rho_\lambda'(s)<0 \,\,\, \forall s\in[\lambda,+\infty).
\]
Hence, the associated $\phi$-function
\[
\phi_\lambda(s)=\int_{\lambda}^{s}\frac{\rho_\lambda(u)}{\sqrt{1-\rho_\lambda^2(u)}}du, \,\,\, s\in [\lambda,+\infty),
\]
is well defined and satisfies $\phi_\lambda'(s)\rightarrow +\infty$ as $s\rightarrow\lambda.$ This, as before,
gives that the corresponding $(f_s,\phi)$-graph $\Sigma'(\lambda)$ is a $W_c$-hypersurface of
$\h^n\times\R$ with boundary $f_{\lambda}(\R^{n-1})\times\{0\}.$
So, proceeding as in the previous proofs, we obtain a properly embedded
$W_c$-hypersurface $\Sigma(\lambda)$ in $\h^n\times\R$
by reflecting $\Sigma'(\lambda)$ with respect to the horizontal hyperplane $\h^n\times\{0\}.$

To prove the assertions (i) and (ii), let us first observe that
\eqref{eq-solutionedohyperbolic} yields
\begin{equation}\label{eq-limittau}
 \lim_{s\rightarrow+\infty}\rho_\lambda(s)=\sqrt{\mathfrak C_n}\,\cdot
\end{equation}
Hence, if $\mathfrak C_n>0,$
\[
\phi_\lambda(s)=\int_{\lambda}^{s}\frac{\rho_\lambda(u)}{\sqrt{1-\rho_\lambda^2(u)}}du\ge \int_{\lambda}^{s}\rho_\lambda(u)du\ge \sqrt{\mathfrak C_n}(s-\lambda),
\]
which implies that $\phi_\lambda$ is unbounded.
Besides, setting $c=Kn(n-1),$
we have $\lambda_0={\rm arctanh}\sqrt{1+K}.$ In this case, it
is easily checked that
\[
\rho_{\lambda_0}(s)=(1+K)\coth(s), \,\,\, s\in[\lambda_0,+\infty).
\]
However, $\rho_{\lambda_0}$ is also a solution of \eqref{eq-edo10} (for $\alpha(s)=\tanh(s)$),
which implies that $\Sigma(\lambda_0)$ has constant sectional curvature $K.$
This proves (i).

If $\mathfrak C_n=0,$ we have that $\tau_\lambda'=a\tau_\lambda,$ which implies that
$2\rho_\lambda'=a\rho_\lambda.$ Thus, observing that $\sup (-1/a)=1/n,$ we have
\[
\phi_\lambda(s)=
\int_{\lambda}^{s}\frac{2\rho_\lambda'(u)}{a(u)\sqrt{1-\rho_\lambda^2(u)}}du\le
\frac{2}{n}\int_{\rho_\lambda(s)}^{1}\frac{d\rho_\lambda}{\sqrt{1-\rho_\lambda^2}}=
\frac{2}{n}\left(\frac{\pi}{2}-\arcsin(\rho_\lambda(s))\right),
\]
which implies that $\phi_\lambda(s)<\pi/n.$ In particular, there exists
$\sigma\in(0,1]$ such that
\[
\sup\phi_\lambda=\sigma\frac{\pi}{n},
\]
which shows (ii) and concludes the proof.
\end{proof}

\begin{figure}
 \centering
 \includegraphics[width=4.5cm]{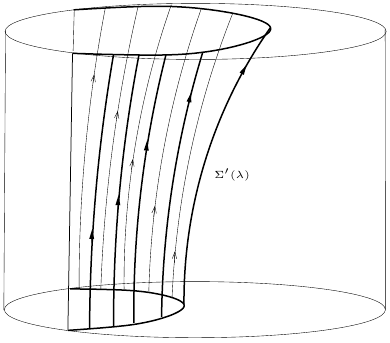}
 \hfill
 \includegraphics[width=5.5cm]{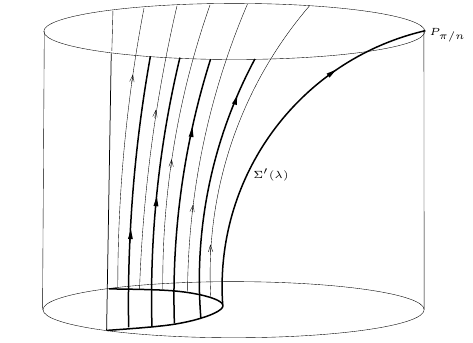}
 \caption{\small A piece of the $W_c$-hypersurface of Theorem \ref{th-hyperbolicscalar} (case (i) on the left, case (ii) on the right).}
 \label{fig-equidistantgraphs}
\end{figure}

\section{Uniqueness of Rotational Elliptic Weingarten Spheres} \label{sec-uniqueness}

As we have pointed out in Remark \ref{rem-MP}, the Maximum Principle
applies to
elliptic Weingarten hypersurfaces of $M\times\R.$ This fact, together with
the main results in \cite{delima} (see also \cite{esp-rosen, oliveira-schweitzer}),
allows us to apply the Alexandrov reflection method to provide
uniqueness results for the rotational
elliptic Weingarten spheres of $\q\times\R$ ($\epsilon\ne 0, n\ge 3$)
we constructed in Section \ref{sec-W-rotational}. Similar results
were obtained in \cite{delima-manfio-santos} for hypersurfaces of constant
higher order mean curvatures.

\begin{definition}
Let $P_0$ be a totally geodesic hypersurface of either hyperbolic space $\h^n$ or
an open hemisphere $\s_+^n$ of \,$\s^n.$ We call the hypersurface $P:=P_0\times\R$
a \emph{vertical hyperplane} of $\q\times\R.$
\end{definition}

Regarding the spherical part of the above definition, we remark that
the Alexandrov method is not
effective for vertical hyperplanes over the whole sphere. However,
it works well for vertical hyperplanes over open hemispheres
(cf. \cite{abresch-rosenberg}, pg. 144).

\begin{theorem}[Jellett--Liebmann-type theorem]\label{th-uniqueness}
For $n\ge3,$ let $\Sigma$ be a {compact} connected
elliptic Weingarten hypersurface of
\,$\mathbb Q_{\epsilon}^n\times\R.$
Then, $\Sigma$ is an {embedded} rotational sphere.
\end{theorem}
\begin{proof}
Since $\Sigma$ is strictly convex and compact,
\cite[Theorems 1 and 2]{delima} apply and give that
$\Sigma$ is embedded and homeomorphic to $\s^n$.
In this way, we can perform Alexandrov reflections on $\Sigma$ with respect to horizontal hyperplanes
$P_t:=\q\times\{t\}$ coming down from above $\Sigma.$ Then,
a standard argument using the maximum principle shows that
$\Sigma$ is symmetric with respect to some horizontal plane $P_{t_0}$.

For $\epsilon=-1,$ we can proceed as above by performing Alexandrov reflections on $\Sigma$
with respect to vertical hyperplanes of $\h^n\times\R$ to conclude that it has
vertical symmetries in all directions. Therefore, $\Sigma$ is rotational.

For $\epsilon=1,$ assuming $t_0=0$ and identifying
$\s^n\times\{0\}$ with $\s^n,$ we have that
$\Sigma_0:=\Sigma\cap\s^n$ is the boundary of
$\pi(\Sigma)$ in $\s^n$ and that $\Sigma\setminus \Sigma_0$
has two connected components, each of which being a graph over
${\rm int}(\pi(\Sigma))$.
By \cite[Lemma 1]{delima}, the second fundamental form of $\Sigma_{0}$,
as a hypersurface of \,$\s^n$, is
positive definite. In particular, $\Sigma_{0}$ is not totally geodesic in \,$\s^n$. Thus,
by \cite[Theorem 1]{docarmo-warner}, $\Sigma_{0}$ is contained in an open hemisphere $\s_+^n$
of $\s^n$, which implies that the same is true for $\pi(\Sigma).$
Indeed, the other option would be
$\s^n\setminus \pi(\Sigma)\subset \s_+^n$, in which case $\Sigma$ would
have at least one concave point, a contradiction with the assumption
it is strictly convex.
Hence, $\Sigma\subset \s_+^n\times\R$, and again
we may apply Alexandrov reflections on the vertical hyperplanes in
$\s_+^n\times\R$ to deduce that  $\Sigma$ is rotational.
\end{proof}

Let us see now that the compactness hypothesis in Theorem \ref{th-uniqueness}
can be replaced by completeness if we add conditions on the height function
$\xi$ of $\Sigma$
and on its second fundamental form. This is accomplished by means of the following
general height estimate, obtained in \cite{delima-manfio-santos}.

\begin{lemma}\cite[Proposition 3]{delima-manfio-santos}\label{lem-heightestimate}
Consider an arbitrary Riemannian manifold \,$M,$ and let $\Sigma\subset M\times\R$ be a
compact vertical graph of a nonnegative function defined on a domain $\Omega\subset M\times\{0\}.$
Assume $\Sigma$ strictly convex up to $\partial\Sigma\subset M\times\{0\}.$ Under these conditions,
the following height estimate holds:
\begin{equation} \label{eq-delta}
\xi(x)\le\frac{1}{\inf_\Sigma k_0} \,\,\, \forall x\in\Sigma,
\end{equation}
where $k_0$ is the least principal curvature function of $\Sigma$.
\end{lemma}

\begin{theorem} \label{th-uniquenessforcomplete}
Assume $n\ge 3$, and let $\Sigma$ be a complete connected
elliptic Weingarten hypersurface of \,$\mathbb Q_\epsilon^n\times\R$
whose height function has a local
extreme point $x\in\Sigma.$
If the least principal curvature $k_0$ of $\Sigma$
is bounded away from zero, then
$\Sigma$ is an embedded rotational sphere.
\end{theorem}

\begin{proof}
As in Theorem~\ref{th-uniqueness}, $\Sigma$ fulfills the hypotheses of \cite[Theorems 1 and 2]{delima}, which implies that
$\Sigma$ is properly embedded and homeomorphic to either
$\s^n$ or $\R^n.$ In the former case, the result follows from
Theorem~\ref{th-uniqueness}, so we can assume that $\Sigma$ is noncompact.
Under this assumption, \cite[Theorems 1 and 2]{delima} also give that
the extreme point $x$ (which we assume to be a maximum)
is unique, and that the height function $\xi$ of $\Sigma$ is unbounded (below).

Given a horizontal hyperplane
$P_t=\q\times\{t\}$ with $t<\xi(x),$ the part $\Sigma_t^+$ of $\Sigma$ which lies above $P_t$
must be a vertical graph with boundary in $P_t$.
If not, for some $t'$ between $t$ and $\xi(x),$ $P_{t'}$ would be orthogonal
to $\Sigma$ at one of its points. Then, the
boundary maximum principle would give that $\Sigma$ is
symmetric with respect to $P_{t'},$ which is impossible, since we are assuming
$\xi$ unbounded, and the closure of $\Sigma_{t'}^+$ in $\Sigma$ is compact.

It follows from the above that, for
$|t|$ sufficiently large, one has
\[\xi(x)-t>\frac{1}{\inf_\Sigma k_0}\ge\frac{1}{\inf_{\Sigma_t^+} k_0}\,,\]
which clearly
contradicts Lemma \ref{lem-heightestimate}. This finishes the proof.
\end{proof}

\begin{corollary}
Any complete strictly convex CSC hypersurface  of \,$\q\times\R$ $(n\ge 3)$
satisfying the hypotheses
of Theorem \ref{th-uniquenessforcomplete}   is
necessarily an embedded rotational sphere  of constant sectional
curvature $K>\epsilon.$
\end{corollary}

\begin{proof}
Let $\Sigma_0\subset\q\times\R$ be a strictly convex hypersurface of
constant scalar curvature $c$ (and so an elliptic $W_c$-hypersurface)
which fulfills the hypotheses of Theorem \ref{th-uniquenessforcomplete}. Then,
$\Sigma_0$ is an embedded rotational $W_c$-sphere. Let
$o\in\Sigma_0$ be the point of least height on $\Sigma_0.$ Clearly,
$\theta^2=1$ at $o.$ This, together with \eqref{eq-scalar0} and the strict
convexity of $\Sigma_0$, implies that
$c>\epsilon n(n-1).$ Also, since $\Sigma_0$ is rotational, in a neighborhood
of $o,$ $\Sigma_0$ is a rotational $W_c$-$(f_s,\phi)$-graph. In particular, its $\rho$ function
is the solution of \eqref{eq-edo3}, for $\alpha=-\cot_\epsilon,$ which satisfies $\rho(0)=0.$ Thus, in a neighborhood
of $o,$ up to an ambient isometry, $\Sigma_0$ coincides with the hypersurface $\Sigma$ of constant scalar curvature
$c>\epsilon n(n-1)$ obtained in Theorem \ref{th-scalar}. Since $\Sigma_0$ and $\Sigma$ are both elliptic, they must coincide.
In particular,  $\Sigma=\Sigma_0$ is compact, i.e., $c>0$ and $\Sigma_0$ is a rotational sphere of constant
curvature $K=c/(n(n-1))>\epsilon.$
\end{proof}

\section{Constant Sectional Curvature Hypersurfaces of $\q\times\R$} \label{sec-sectionalcurvature}

In this final section, we apply the methods and results developed in
the previous ones to give
new proofs of the main theorems of \cite{manfio-tojeiro}. There, for $n\ge 3,$
the authors construct and classify the constant sectional curvature hypersurfaces of $\q\times\R$
$(\epsilon\ne 0).$
Their proofs rely on parametrizations of symmetric hypersurfaces of $\q\times\R,$
as introduced in \cite{dillenetal}, as well as
on the main result of that paper. Our proofs, instead, are coordinate-free and more direct.
Another distinction of our approach is the simple way we show that constant sectional curvature
hypersurfaces of $\q\times\R$ with non vanishing $T$-field have the $T$-property. This fact, which is also
proved and nicely used in \cite{manfio-tojeiro}, plays a fundamental role here, as we shall see.

First, let us observe  that, if $\Sigma:=\Sigma_0\times\R$ is a symmetric
cylinder over a hypersurface $\Sigma_0$ of $\q,$ then
$\Sigma_0$ is an open set of a geodesic sphere, a horosphere or an
equidistant hypersurface. Thus, from Gauss equation \eqref{eq-gauss},
we have that $\Sigma$ has constant sectional curvature if and only if
$\Sigma_0$ is contained in a horosphere. If so, $\Sigma$ is
necessarily  flat. We also point out that, by the first equality in
\eqref{eq-scalargraph},
there is no flat parabolic $(f_s,\phi)$-graph in $\q\times\R.$

\begin{theorem} \label{th-sectionalcurvature01}
Given $n\ge 3,$   let $\Sigma$ be a connected symmetric hypersurface of \,$\q\times\R$ with constant sectional
curvature $K.$ Then, $K\ge\epsilon$ and $\Sigma$ is an open set of one of the following
properly embedded hypersurfaces:
\begin{itemize}[parsep=1ex]

 \item[\rm i)] One of the rotational  hypersurfaces of constant sectional curvature $K>\epsilon$ of Theorem \ref{th-scalar}.
 \item[\rm ii)] One of the translational hypersurfaces of constant sectional curvature $K\in(-1,0)$ of Theorems \ref{th-parabolicscalar}-(i)
 and \ref{th-hyperbolicscalar}-(i).
 \item[\rm iii)] A horizontal hyperplane of constant sectional curvature $K=\epsilon$.
 \item[\rm iv)] A flat vertical cylinder over a horosphere.
\end{itemize}
\end{theorem}
\begin{proof}
Let us suppose that $\theta T$ never vanishes on $\Sigma.$ In this case,
$\Sigma$ is given by a union of vertical graphs with no critical points.
Let $\Sigma'$ be one of these graphs.
Then, since $\Sigma$ is symmetric, up to a reflection over a horizontal hyperplane,
$\Sigma'$ is an $(f_s,\phi)$-graph over a family $\mathscr F$ of parallel totally umbilical hypersurfaces of $\q$
(recall that $\phi$ is supposed to be increasing).

As we have seen in Section \ref{sec-symmetricCSC},
the $\rho$ function of $\Sigma'$ satisfies the ODE
\begin{equation} \label{eq-odesectional}
\alpha\rho\rho'+(\alpha^2+\epsilon)\rho^2=0,
\end{equation}
where $\alpha(s)$ is the principal curvature of $f_s$\,.
Also, from the first equality in \eqref{eq-scalargraph},
we have that $K>\epsilon$ and that any initial condition
$\rho_0=\rho(s_0)$ is determined by $K$ and $\alpha(s_0).$

Since $\Sigma$ is connected, it is
either rotational or translational. Assuming the former,
we have that  $\Sigma'$ is rotational, so that
$\alpha=-\coth_\epsilon$ and  the solution $\rho$ of
\eqref{eq-odesectional} is $\rho(s)=C\tan_\epsilon(s)$ for some constant $C>0.$
However, since $C$ is determined by $K$ and $\alpha(s_0),$
this solution must coincide with the one defined
in \eqref{eq-rhoconstantseccurvature}, that is, $$\rho(s)=(K-\epsilon)\tan_\epsilon(s).$$
In particular, up to an ambient isometry, $\Sigma'$ is contained in  the properly embedded
CSC hypersurface of Theorem \ref{th-scalar}, say  $\tilde\Sigma,$ that
has constant sectional curvature $K.$ (Notice that $\tilde\Sigma$ is built on an
$(f_s,\phi)$-graph whose $\rho$-function is  $\tilde\rho(s)=(K-\epsilon)\tan_\epsilon(s)=\rho(s).$)
Since $\Sigma'$ is arbitrary and $\Sigma$ is connected,
we have that $\Sigma\subset\tilde\Sigma.$

If $\Sigma$ is translational, it is either
parabolic or hyperbolic. In any case,
we can argue as in the preceding paragraph
and conclude that $\Sigma$ is contained in one
of the translational hypersurfaces of Theorems \ref{th-parabolicscalar}-(i)
and \ref{th-hyperbolicscalar}-(i).

Now, suppose that $T$  vanishes on an open set $\mathcal O$ of $\Sigma.$ Thus,
$\mathcal O$ is contained in a (totally geodesic) horizontal hyperplane. Gauss equation
then gives that $\mathcal O,$ and so $\Sigma,$ has constant sectional curvature $K=\epsilon.$
This implies that $\Sigma=\mathcal O.$ Otherwise, there would be a point in $\Sigma$ at
which $\theta T\ne 0.$ But, as we have seen, $K>\epsilon$ at such a point.

Analogously, assume that $\theta$ vanishes on an open set $\mathcal O$ of $\Sigma.$
In this case, as we pointed out, $\mathcal O$ is a flat cylinder over an open set
of a horosphere of $\h^n.$ In particular, $\Sigma$ is flat and parabolic.
Since a parabolic $(f_s,\phi)$-graph cannot be flat,
$\theta T$ must vanish on $\Sigma,$ which implies that $\Sigma=\mathcal O.$

Finally, let us suppose that neither $T$ nor $\theta$ vanishes on an open set
of $\Sigma.$ Under this assumption, the set $\mathcal O\subset\Sigma$ on which $\theta T$ never vanishes
is open and dense in $\Sigma.$ However, from the first part of the proof,
any connected component of $\mathcal O$ is contained
in a fixed hypersurface $\tilde\Sigma$ (from one of the
Theorems \ref{th-scalar}, \ref{th-parabolicscalar}-(i) or \ref{th-hyperbolicscalar}-(i)),
which implies that the same is true for $\Sigma.$
This concludes the proof.
\end{proof}

Given a hypersurface $\Sigma$ of $\q\times\R$, let
$\{X_1\,, \dots ,X_n\}\subset T\Sigma$ be an orthonormal frame of its principal directions.
It follows from Gauss equation \eqref{eq-gauss001} that
\begin{equation}\label{eq-ricci02}
{\rm Ric}(X_i\,,X_j)=\sum_{k=1}^{n}\langle\overbar R(X_k,X_i)X_j,X_k\rangle+H\delta_{ij}k_i-\delta_{ij}k_i^2\,,
\end{equation}
where $k_1\,,\dots ,k_n$ are the corresponding principal curvatures of $\Sigma,$
and ${\rm Ric}$ denotes its \emph{Ricci tensor}, which we define as
\[
{\rm Ric}(X,Y):={\rm trace}(Z\mapsto R(Z,X)Y), \,\,\, X,Y\in T\Sigma.
\]

The following lemma, which has its own interest, gives a simple characterization
of the hypersurfaces of $\q\times\R$ having the $T$-property.
(Recall that a hypersurface $\Sigma\subset M\times\R$ is said to have the $T$-property
if $T$ is a principal direction at any  of its points.)

\begin{lemma} \label{lem-T}
Given $n\ge 3,$  let $\Sigma$ be a hypersurface of \,$\q\times\R$ with
non vanishing $T$-field. Then, $\Sigma$ has the
$T$-property if and only if its principal directions $X_1\,, \dots ,X_n$ diagonalize its Ricci tensor.
Consequently, if $\Sigma$ is an Einstein hypersurface (in particular, if $\Sigma$ has constant sectional curvature),
then it has the $T$-property.
\end{lemma}

\begin{proof}
Choosing $i, j, k\in\{1,\dots ,n\}$ with $i\ne j\ne k\ne i,$ it follows
from \eqref{eq-barcurvaturetensor} that
\[
\langle\overbar R(X_k,X_i)X_j,X_k) = -\epsilon\langle X_i,T\rangle\langle X_j,T\rangle.
\]
Combining this equality with \eqref{eq-ricci02}, we get
\[
{\rm Ric}(X_i\,,X_j)=\sum_{k=1}^{n}\langle\overbar R(X_k,X_i)X_j,X_k)=-\epsilon(n-2)\langle X_i\,, T\rangle\langle X_j\,, T\rangle,
\]
from which the result follows.
\end{proof}

Let us show now that,
setting
\[
I_\epsilon:=\left\{
\begin{array}{rcl}
(0,1)         & \text{if} & \epsilon=\phantom-1\\[1ex]
(-1,0)        & \text{if} & \epsilon=-1,
\end{array}
\right.
\]
for all $K\in I_\epsilon,$
there exists a nonsymmetric constant sectional curvature
hypersurface $\Sigma_K$ in $\mathbb Q_\epsilon^3\times\R$ whose angle function is constant,
and whose sectional curvature is $K$ (see also \cite[Section 6]{manfio-tojeiro}).
For its construction, one has just
to consider a parallel family $\mathscr F$ of embeddings $f_s:M_0\rightarrow\mathbb Q_\epsilon^3,$
$s\in (-\delta,\delta),$ such
that $f_0$ is non umbilical and flat with the induced metric, that is,
$k_1^0(p)k_2^0(p)=-\epsilon$ for all $p\in M_0$\,. Then, the classical formula for the principal curvatures of the
parallel hypersurfaces $f_s$ gives that all $f_s$ are, in fact, flat and non umbilical, so that
\[
k_1^s(p)k_2^s(p)=-\epsilon \,\,\, \forall p\in M_0, \,\, s\in(-\delta,\delta).
\]
Now, consider the $(f_s,\phi)$-graph $\Sigma_K$ with constant
$\rho$-function $\rho=\sqrt{1-\epsilon K}.$ Then, $\Sigma_K$ has
constant angle $\theta=\sqrt{1-\rho^2}.$ Besides,
the principal curvatures of
$\Sigma_K$ are
$k_1=-\rho k_1^s,$ $k_2=-\rho k_2^s,$ and $k_3=0.$ Hence,
from Gauss equation \eqref{eq-gauss},
$\Sigma_K$ has indeed constant sectional curvature $K.$

\begin{definition}
We shall call such a $\Sigma_K\subset\mathbb Q_\epsilon^3\times\R$ a \emph{flat-foliated graph.}
\end{definition}

Our next and final result, together with Theorem \ref{th-sectionalcurvature01}, provides
a full classification of the constant sectional curvature hypersurfaces of
$\q\times\R$ $(n\ge 3)$.

\begin{theorem}\label{th-sectionalcurvature02}
Let $\Sigma$ be a connected hypersurface of \,$\q\times\R$ with constant sectional curvature $K.$
If $n>3,$ $\Sigma$ is symmetric.  If $n=3,$ $\Sigma$ is either symmetric or
a nonsymmetric hypersurface which is locally a
constant angle flat-foliated graph $\Sigma_K.$
\end{theorem}

\begin{proof}

Proceeding as in the proof of Theorem \ref{th-sectionalcurvature01}, let us assume first
that $\theta T$ is nowhere vanishing on $\Sigma,$ in which case $\Sigma$ is a union of
vertical graphs with no critical points.

Let $\Sigma'\subset\Sigma$ be a vertical graph.
By Lemma \ref{lem-T}, $T$ is a principal direction of $\Sigma'$. Hence,
\cite[Theorem 1]{tojeiro} (see also \cite[Theorem 6]{delima-roitman}) applies and
gives that $\Sigma'$ is an $(f_s,\phi)$-graph over a
family $\mathscr F$ of parallel hypersurfaces of $\q.$ Thus,
by combining the identities \eqref{eq-principalcurvatures} with Gauss equation
\eqref{eq-gauss}, we conclude that
the $\rho$ function of $\Sigma'$ satisfies
\begin{equation}\label{eq-scalargraph02}
\begin{aligned}
 K &=k_i^s(p)k_j^s(p)\rho^2(s)+\epsilon \quad (i\ne j=1,\dots, n-1). \\
 K &=-k_i^s(p)\rho(s)\rho'(s)+\epsilon(1-\rho^2(s)) \quad (i=1,\dots, n-1). \\
 \end{aligned}
\end{equation}

If $n>3,$ it follows easily from the first of the equations in \eqref{eq-scalargraph02}
that $k_i^s(p)$ is independent of $i$ and $p.$ Hence,
each $f_s\in\mathscr F$ is totally umbilical, which implies that  $\Sigma'$
is symmetric.  Therefore, $\Sigma$ is symmetric, since $\Sigma'\subset\Sigma$ is arbitrary.

Suppose now that $n=3.$ If  $f_s\in\mathscr F$ is totally umbilical for any graph
$\Sigma'\subset\Sigma,$ as above, we have that $\Sigma$ is symmetric.
So, assume that there is $\Sigma'\subset\Sigma$ such that
$f_s$ is non totally umbilical,
so that $\Sigma',$ and so $\Sigma,$ is nonsymmetric.
In this case, the $\rho$ function of $\Sigma'$ is constant. Otherwise,
from the second equation in \eqref{eq-scalargraph02},  $f_s$ would be totally
umbilical. From this same equation, we have that
$\Sigma'$  has constant sectional curvature $K=\epsilon(1-\rho^2)=\epsilon\theta^2.$
So, from the first equality in \eqref{eq-scalargraph02}, we have
$k_i^s(p)k_j^s(p)=-\epsilon,$ which implies that each $f_s$ is flat, that is,
$\Sigma'$ is a constant angle flat-foliated graph $\Sigma_K.$

The above reasoning  shows that, in a neighborhood of a point at which
$\theta T\ne 0,$ $\Sigma$ is either a symmetric $(f_s,\phi)$-graph or a
flat-foliated graph. In the former case, as we know, $K>\epsilon.$ In the latter
case, $K=\epsilon\theta^2\ne\epsilon.$ Therefore, we can argue just as in the proof of
Theorem \ref{th-sectionalcurvature01} to show that
$\Sigma$ is contained in  a horizontal hyperplane if $T$
vanishes on an open set of $\Sigma.$
Since neither parabolic $(f_s,\phi)$-graphs nor flat-foliated graphs are flat,
we can also argue as  in the proof of
Theorem \ref{th-sectionalcurvature01} to show that
$\Sigma$ is contained in  a cylinder over a horosphere if $\theta$
vanishes on an open set of $\Sigma.$

Finally, if neither $T$ nor $\theta$ vanishes in an open set $\Sigma,$ the set
$\mathcal O\subset\Sigma$ of points where $\theta T$ never vanishes is open and dense
in $\Sigma.$ Since the theorem is valid for any connected component of $\mathcal O,$ it follows
that it is valid for $\Sigma$ as well.
\end{proof}

\end{document}